%% file: compatible-hamilton-cycles-dirac.tex
\documentclass[11pt]{article}

\usepackage{amsmath, amssymb, amsthm}
\usepackage{comment}

\usepackage{tikz}
\usetikzlibrary{arrows}

\theoremstyle{plain}
\newtheorem{thm}{Theorem}[section]
\newtheorem*{thm*}{Theorem}
\newtheorem{prop}[thm]{Proposition}
\newtheorem{lem}[thm]{Lemma}
\newtheorem*{lem*}{Lemma}
\newtheorem{defn}[thm]{Definition}

\newtheorem{conj}[thm]{Conjecture}

\date{}
\title{\vspace{-0.85cm}Compatible Hamilton cycles in Dirac graphs}

\author{
Michael Krivelevich \thanks{School of Mathematical Sciences, Raymond and Beverly Sackler Faculty of Exact Sciences, Tel Aviv University, Tel Aviv,
69978, Israel. krivelev@post.tau.ac.il.
Research supported in part by USA-Israel BSF Grant 2010115 and by grant 912/12 from the Israel Science Foundation.}
\and
Choongbum Lee \thanks{Department of Mathematics,
MIT, Cambridge, MA 02139-4307. Email: cb\_lee@math.mit.edu.
Research supported in part by NSF Grant DMS-1362326.}
\and
Benny Sudakov \thanks{Department of Mathematics, ETH, 8092 Zurich, Switzerland. Email: benjamin.sudakov@math.ethz.ch. Research supported in part by SNSF grant 200021-149111 and by a USA-Israel BSF grant.}
}

\begin{document}

\maketitle

\begin{abstract}
A graph is Hamiltonian if it contains a cycle passing through every vertex
exactly once. A celebrated theorem of Dirac from 1952 asserts that every graph on
$n\ge 3$ vertices with minimum degree at least $n/2$ is Hamiltonian. We refer to such graphs as Dirac graphs.
In this paper we obtain the following strengthening of this result.
Given a graph $G=(V,E)$, an {\em incompatibility system} $\mathcal{F}$ over $G$ is a family $\mathcal{F}=\{F_v\}_{v\in V}$ such that for every $v\in V$, the set $F_v$ is a family of unordered pairs $F_v \subseteq \{\{e,e'\}: e\ne e'\in E, e\cap e'=\{v\}\}$.
An incompatibility system is {\em $\Delta$-bounded} if for every vertex $v$
and an edge $e$ incident to $v$, there are at most $\Delta$ pairs in $F_v$
containing $e$.
We say that a cycle $C$ in $G$ is {\em compatible} with $\mathcal{F}$
if every pair of incident edges $e,e'$ of $C$ satisfies $\{e,e'\} \notin F_v$, where $v=e\cap e'$. This notion is partly motivated by
a concept of transition systems defined by Kotzig in 1968, and can be viewed as a quantitative measure of
robustness of graph properties.
We prove that there is a constant $\mu>0$ such that
for every $\mu n$-bounded incompatibility system $\mathcal{F}$ over a Dirac graph $G$,
there exists a Hamilton cycle compatible with $\mathcal{F}$.
This settles in a very strong form, a conjecture of
H\"{a}ggkvist from 1988.
\end{abstract}

\section{Introduction}

A \emph{Hamilton cycle} in a graph $G$ is a cycle passing through each vertex of
$G$, and a graph is \emph{Hamiltonian} if it contains
a Hamilton cycle. Hamiltonicity, named after Sir Rowan Hamilton who studied it in the 1850s, is a very important
and extensively studied concept in graph theory.
The study of Hamiltonicity has mainly been concerned with looking for simple sufficient conditions
implying Hamiltonicity. One of the most important results in this direction
is Dirac's theorem asserting that every $n$-vertex graph, $n\ge 3$, of minimum
degree at least $\frac{n}{2}$ contains a Hamilton cycle.  In this context,
we define a {\em Dirac graph} as an $n$-vertex graph of minimum
degree at least $\frac{n}{2}$.
Note that the bound $\frac{n}{2}$ is tight,
as can be seen by the following two examples:
first is a graph obtained by taking two vertex-disjoint complete
graphs of order $k$ and identifying one vertex from each of them,
and second is the complete bipartite graph with parts of sizes
$k$ and $k-1$. Both graphs have $2k-1$ vertices and minimum
degree $k-1$, but are not Hamiltonian.


Recently there has been an increasing interest in studying
robustness of graph properties, aiming to strengthen classical results
in extremal and probabilistic combinatorics.
For example, consider the Hamiltonicity property of Dirac graphs.
Then one can ask, ``How strongly
do Dirac graphs possess the property of being Hamiltonian?''.
There are several ways to answer this question using
different measures of robustness. For example, one can try to show
that a Dirac graph has many Hamilton cycles or that Maker can win a
Hamiltonicity game played on edges of a Dirac graph.
These extensions and other similar questions have been answered in \cite{CuKa}
for the number of Hamilton cycles, in \cite{ChKuOs} and \cite{FeKrSu} for
the number of edge-disjoint Hamilton cycles, in \cite{Heinig} for
the cycle space generated by Hamilton cycles, and in \cite{KrLeSu14} for
Hamiltonicity of random subgraphs and for the Maker-Breaker games on Dirac graphs. Also, very recently,
a number of related important problems on regular Dirac graphs, such
as the existence of decomposition of its edge set into
Hamilton cycles, have been settled in a series
of papers starting from \cite{KuLoOsTr}, using a structural result
proved in \cite{KuLoOsSt}. In fact, the study of robustness
of graph properties can be identified as one of the central themes in
extremal graph theory, and its scope extends far beyond
Hamiltonicity and Dirac graphs (see, e.g., \cite{SuVu}).

\medskip

In this paper, we are interested in yet another type of robustness, and study
the robustness of Hamiltonicity of Dirac graphs with respect to a new measure.

\begin{defn}
Let $G=(V,E)$ be a graph.

\vspace{-0.2cm}
\begin{itemize}
  \setlength{\itemsep}{0pt} \setlength{\parskip}{0pt}
  \setlength{\parsep}{0pt}
\item[(i)] An {\em incompatibility system} $\mathcal{F}$ over $G$ is a family $\mathcal{F}=\{F_v\}_{v\in V}$ such that for every $v \in V$, $F_v$ is a family of unordered pairs $F_v \subseteq \{\{e,e'\}: e\ne e'\in E, e\cap e'=\{v\}\}$.

\item[(ii)] If $\{e,e'\} \in F_v$ for some edges $e,e'$ and vertex $v$, then we
say that $e$ and $e'$ are \emph{incompatible} in $\mathcal{F}$.
Otherwise, they are \emph{compatible} in $\mathcal{F}$. A subgraph $H \subseteq G$
is \emph{compatible} in $\mathcal{F}$, if all its pairs of edges $e$ and $e'$ are compatible.
\item[(iii)]  For a positive integer $\Delta$, an incompatibility system $\mathcal{F}$
is {\em $\Delta$-bounded} if for each vertex $v \in V$ and an edge $e$ incident to $v$, there are at most $\Delta$ other edges $e'$ incident to $v$ that are
incompatible with $e$.
\end{itemize}
\end{defn}

The definition is motivated by two concepts in graph theory.
First, it generalizes \emph{transition systems} introduced by Kotzig \cite{Kotzig68}
in 1968, where a transition system is a $1$-bounded
incompatibility system. Kotzig's work was motivated by a problem of Nash-Williams on cycle covering
of Eulerian graphs (see, e.g. Section 8.7 of \cite{Bondy95}).

Incompatibility systems and compatible Hamiton cycles also generalize
the concept of properly colored Hamilton cycles in edge-colored graphs,
The problem of finding properly colored Hamilton cycles in edge-colored
graph was first introduced by Daykin \cite{Daykin}. He asked if there
exists a constant $\mu$ such that for large enough $n$,
there exists a properly colored Hamilton cycle in every
edge-coloring of a complete graph $K_{n}$ where each vertex has at
most $\mu n$ edges incident to it of the same color (we refer to
such coloring as a \emph{$\mu n$-bounded edge coloring}).
Daykin's question has been answered
independently by Bollob\'as and Erd\H{o}s \cite{BoEr76} with $\mu=1/69$,
and by Chen and Daykin \cite{ChDa} with $\mu=1/17$. Bollob\'as
and Erd\H{o}s further conjectured that all $(\lfloor \frac{n}{2}\rfloor-1)$-bounded
edge coloring of $K_{n}$ admits a properly colored Hamilton cycle.
After subsequent improvements by Shearer \cite{Shearer} and by Alon and
Gutin \cite{AlGu}, Lo \cite{Lo12} recently settled the conjecture asymptotically,
proving that for any positive $\varepsilon$, every $(\frac{1}{2}-\varepsilon)n$-bounded
edge coloring of $E(K_n)$ admits a properly colored Hamilton cycle.

Note that a $\mu n$-bounded edge coloring naturally defines a $\mu n$-bounded
incompatibility system, and thus the question mentioned above can
be considered as a special case of the problem of finding compatible
Hamilton cycles. However, in general, the restrictions introduced
by incompatibility systems need not come from an edge-coloring of graphs,
and thus results on properly colored Hamilton cycles
do not necessarily generalize easily to incompatibility systems.

\medskip{}

In this paper we study compatible Hamilton cycles in Dirac graphs.
Our work is motivated by the following conjecture of H\"aggkvist
from 1988 (see \cite[Conjecture 8.40]{Bondy95}).

\begin{conj} \label{conj:haggkvist}
Let $G$ be a Dirac graph. For every $1$-bounded incompatibility
system $\mathcal{F}$ over $G$, there exists a Hamilton
cycle compatible with $\mathcal{F}$.
\end{conj}

Here, we settle this conjecture, in fact in a very strong form.

\begin{thm} \label{thm:main}
There exists a constant $\mu$ such that the following holds
for large enough $n$. For every $n$-vertex Dirac graph $G$
and a $\mu n$-bounded incompatibility system $\mathcal{F}$ defined
over $G$, there exists a Hamilton cycle in $G$ compatible with $\mathcal{F}$.
\end{thm}

Our theorem shows that Dirac graphs are very robust against
incompatibility systems, i.e., one can find a Hamilton cycle even after
forbidding a quadratic number of pairs of edges incident to each vertex
from being used together in the cycle.
The order of magnitude is clearly best possible since we can simply
forbid all pairs incident to some vertex from being used together
to disallow a compatible Hamilton cycle.
However, it is not clear what the best possible value of $\mu$
should be. Our proof shows the existence of a positive constant $\mu$ (approximately
$10^{-16}$ although no serious attempt has been made to optimize
the constant), and a variant of a construction of
Bollob\'as and Erd\H{o}s \cite{BoEr76} shows that
$\mu \le \frac{1}{4}$. See Section \ref{sec:conclusion} for
further discussion.

\medskip

The rest of the paper is organized as follows.
In Section \ref{sec:rotation_extension}, we discuss and
develop our main tool, based on P\'osa's rotation-extension technique.
The proof of Theorem \ref{thm:main} consists of several
cases, and will be given in Sections \ref{sec:proof_main}
and \ref{sec:proof_sub}. In Section \ref{sec:conclusion},
we conclude this paper with some remarks.

\medskip

\noindent \textbf{Notation.} A graph $G = (V,E)$ is given by a
pair of its vertex set $V = V(G)$ and edge set $E = E(G)$.
For a set $X$, let $e(X)$ be the number of edges whose both
endpoints are in $X$.
For a pair $X,Y$ of sets of vertices, we define $e(X,Y)$ as the
number of edges with one endpoint in $X$ and the other in $Y$
counted with multiplicity. Hence, $e(X,X) = 2e(X)$.

For a path $P$ (or a cycle $C$), we let $|P|$ (or $|C|$, respectively) be the
number of vertices in the path (cycle, respectively). Also, we define the length
of a path, or a cycle, as its number of edges.
Throughout the paper, we assume that the number of vertices $n$ is
large enough, and omit floor and ceiling signs whenever these are
not crucial.

\section{Rotation-extension technique and its modification}
\label{sec:rotation_extension}

\subsection{P\'osa's rotation-extension technique}

Our main tool in proving Theorem \ref{thm:main}
is P\'osa's rotation-extension technique,
which first appeared in \cite{Posa} (see also \cite[Ch. 10, Problem 20]{Lovasz}).

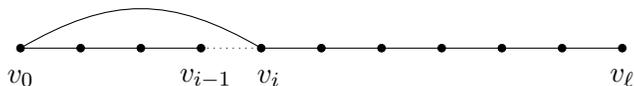
\begin{figure}[h!]
  \centering
  \input{fig-rotatepath}
  \caption{Rotating a path.}
  \label{fig_rotation}
\end{figure}

Let $G$ be a connected graph and let $P = (v_0, \ldots, v_\ell)$ be
a path on some subset of vertices of $G$. Consider the neighbors of $v_0$.
If there exists an edge $\{v_0, w\}$ for $w \notin V(P)$, then
we can \emph{extend} $P$ to find a longer path $(w, v_0, \ldots, v_\ell)$.
Otherwise, if $\{v_0, v_\ell\}$ is an edge of the graph, then
we can use it to close $P$ into a cycle, in which case
since $G$ is connected, we either get a Hamilton cycle, or
can find a  path longer than $P$.

Now assume that we cannot directly extend $P$ as above, and that
$G$ contains an edge $\{v_0, v_i\}$ for some $i$.
Then $P' = (v_{i-1}, \ldots, v_0, v_i, v_{i+1}, \ldots, v_{\ell})$
forms another path of length $\ell$ in $G$ (see Figure
\ref{fig_rotation}). We say that $P'$ is obtained from $P$ by a
\emph{rotation} with {\em fixed endpoint} $v_\ell$, \emph{pivot point}
$v_i$, and {\em broken edge} $(v_{i-1}, v_{i})$. Note that after
performing this rotation, we can now close a cycle of length $\ell$
also using the edge $\{v_{i-1}, v_{\ell}\}$ if it exists in $G \cup P$.
As we perform more and more rotations, we will get more such
candidate edges (call them \emph{closing edges}). The
rotation-extension technique is employed by repeatedly rotating the
path until one can extend the path, or find a closing edge
in the graph to find a cycle (in which case we either find a Hamilton cycle,
or can proceed further by finding a longer path).

Let $P''$ be a path obtained from $P$ by several rounds of
rotations. An important observation that we will use later is that
for every interval $I = (v_j, \ldots, v_k)$ of vertices of $P$ ($1
\le j < k \le \ell$), if no edges of $I$ were broken during these
rotations, then $I$ appears in $P''$ either exactly as it does in
$P$, or in the reversed order. We define the {\em orientation}, or
{\em direction}, of a path $P''$ with respect to an interval $I$ to
be {\em positive} in the former situation, and {\em negative} in the
latter situation.

\subsection{The class of smoothly compatible paths}

There are several difficulties in naively applying the rotation-extension
technique to find compatible Hamilton cycles.
Suppose that we are given a graph $G$ and an incompatibility system
$\mathcal{F}$ over $G$. Let $P=(v_{0},v_{1},\ldots,v_{\ell})$
be a path compatible with $\mathcal{F}$.
First of all, even if $v_{0}$ has a neighbor outside $P$, we might
not be able to extend the path $P$ to a longer path, since
the edge connecting $v_{0}$ to its neighbor outside $P$ can be
incompatible with the edge $\{v_0, v_1\}$.
Second, even if all the neighbors of $v_{0}$ are
in $P$, we might not be able to perform a single round of rotation
since the pair of edges $\{v_{0},v_{i}\}$ and $\{v_{i},v_{i+1}\}$
can be incompatible for each $i$.

Note, however, that the first problem is less of a issue, since
we can find a longer path compatible with $\mathcal{F}$ as long as
$v_0$ has greater than $\mu n$ neighbors outside $P$.
The second problem is more serious, and to overcome
the difficulty, we consider only a special type of paths.
First, we identify the problematic vertices with respect to a given path.

\begin{defn}
Let $\gamma$ be a positive real.
Let $P=(v_{0},v_{1},\ldots,v_{\ell})$ be a path
and $w\in V$ be a vertex ($w$ need not be in $V(P)$).
\vspace{-0.2cm}
\begin{itemize}
  \setlength{\itemsep}{0pt} \setlength{\parskip}{0pt}
  \setlength{\parsep}{0pt}
\item[(i)] A vertex $v_{i}\in V(P)$ is a \emph{bad neighbor} of $w$ \emph{in
$P$}, if $\{w, v_i\}$ is incompatible with $\{v_i, v_{i-1}\}$
or $\{v_i, v_{i+1}\}$. A vertex $v_{i} \in V(P)$ is a \emph{good neighbor} of $w$ \emph{in $P$} otherwise.
\item[(ii)] $w$ is \emph{$\gamma$-bad }for $P$ if there are at least $\gamma|P|$
bad neighbors of $w$ in $P$, and $w$ is {\em $\gamma$-good} for $P$
otherwise.
\item[(iii)] Similarly define {\em bad neighbors}, \emph{$\gamma$-bad} and \emph{$\gamma$-good}
vertices with respect to cycles.
\end{itemize}
\end{defn}

The definition above was made with the intention of forcing
both endpoints of the path to be $\gamma$-good (for an appropriately
chosen $\gamma$), throughout the process of rotation and extension, hoping
to resolve the latter problem mentioned above.
Indeed, suppose that $v_{0}$ is $\gamma$-good for $P$. Then by the rotation-extension
technique, all but at most $\gamma|P|$ of the neighbors of $v_{0}$ in $P$
can be used as a pivot point to give another
path $P'$ compatible with $\mathcal{F}$. This (weakly) resolves
the issue mentioned above regarding the rotation-extension
technique. Indeed, we can perform one round of rotation,
but afterwards, we have no
guarantee that the two new endpoints are good with respect to the
new path $P'$ (the new endpoint might be a bad vertex, or
the number of bad neighbors of the fixed endpoint
might have increased).
In order to resolve this issue, we make the following definition.

\begin{defn}
Let $\gamma$ be a positive real.
Two vertices $v_{1}$ and $v_{2}$ are {\em $\gamma$-correlated}
if there are at least $\gamma n$ vertices $w$ such that
the pair of edges $\{v_1, w\}$ and $\{v_2, w\}$ are incompatible.
We say that $v_{1}$ and $v_{2}$ are {\em $\gamma$-uncorrelated}
if they are not $\gamma$-correlated.
\end{defn}

Note that two vertices being $\gamma$-correlated is a global
condition; it does not depend on individual paths.
We now define a special type of paths and cycles
by utilizing the two concepts defined above.

\begin{defn}
Let $G$ be an $n$-vertex graph with a $\mu n$-bounded
incompatibility system $\mathcal{F}$.
A path $P=(v_{0},v_{1},\cdots,v_{\ell})$ is \emph{smoothly compatible
with $\mathcal{F}$} (or \emph{smooth} in short if $\mathcal{F}$ is
clear from the context) if
\vspace{-0.2cm}
\begin{itemize}
  \setlength{\itemsep}{0pt} \setlength{\parskip}{0pt}
  \setlength{\parsep}{0pt}
\item[(i)] $P$ is compatible with $\mathcal{F}$,
\item[(ii)] both endpoints $v_{0}$ and $v_{\ell}$ are $8\sqrt{\mu}$-good
for $P$,
\item[(iii)] the pair of vertices $v_{0}$ and $v_{\ell}$ is $\sqrt{\mu}$-uncorrelated.
\end{itemize}
\end{defn}

Note that we made the two endpoints to be $8\sqrt{\mu}$-good,
instead of making them $\sqrt{\mu}$-good.
This choice was made for technical reasons, and will be crucial
later.
We conclude this subsection with a proposition asserting that
good vertices and uncorrelated pairs are abundant, thus
showing that the definition of smooth paths is
not too restrictive.

\begin{prop} \label{prop:bound_bad_correlated}
Let $G$ be an $n$-vertex graph with a $\mu n$-bounded
incompatibility system $\mathcal{F}$.
\vspace{-0.2cm}
\begin{itemize}
  \setlength{\itemsep}{0pt} \setlength{\parskip}{0pt}
  \setlength{\parsep}{0pt}
\item[(i)] For every path $P$, there are at most $\sqrt{\mu}n$
vertices that are $2\sqrt{\mu}$-bad for $P$.
The same statement holds for a cycle $C$.
\item[(ii)]For every vertex $v$, there are at most $\sqrt{\mu}n$
vertices that are $\sqrt{\mu}$-correlated with $v$.
\end{itemize}
\end{prop}
\begin{proof}
(i) We prove the claim only for paths as the claim for cycles follows from
the same proof. Let $P=(v_{0},v_{1},\cdots,v_{\ell})$. We count the
number of edges of the form $\{w,v_{i}\}$ such that
$v_{i}$ is a bad neighbor of $w$ in $P$ in two different ways.

First, for a fixed vertex $v_{i}$, since $\mathcal{F}$ is
$\mu n$-bounded, there are at most $2\mu n$
vertices $w$ for which $v_{i}$ is a bad neighbor of $w$;
at most $\mu n$ vertices for which $\{w, v_i\}$ is incompatible with $\{v_i, v_{i+1}\}$,
and at most $\mu n$ vertices for which $\{w, v_i\}$ is incompatible with
$\{v_i, v_{i-1}\}$.
Hence there are at most $|P|\cdot2\mu n$ such edges.
Second, if we define $N$ as the number of vertices with at
least $2\sqrt{\mu}|P|$ bad neighbors in $P$, then by definition,
the number of such edges is at least $N\cdot 2\sqrt{\mu}|P|$.
By combining the two bounds, we see that $N\le \sqrt{\mu}n$.

\medskip
\noindent (ii) Let $M$ be the number of vertices that are $\sqrt{\mu}$-correlated
with $v$. We count the number of pairs of vertices $v', w$ such that
$\{v,w\}$ and $\{v',w\}$ are incompatible
in two different ways. On one hand, each vertex that is $\sqrt{\mu}$-correlated
with $v$ gives at least $\sqrt{\mu}n$ such pairs, and thus
the number of pairs is at least $M\cdot\sqrt{\mu}n$. On the
other hand, for each fixed vertex $w$, there are at most $\mu n$
vertices $v'$ such that $\{v,w\}$ and $\{v',w\}$ are incompatible.
Hence the number of pairs is at most $n\cdot \mu n$.
Therefore $M\le\sqrt{\mu}n$.
\end{proof}

\subsection{Rotating smooth paths}

Our first lemma shows how the rotation part of P\'osa's rotation-extension
technique extends to the class of smooth paths.
This lemma is the most important building block of our proof.

\begin{lem} \label{lem:rotation_1}
Let $\mu$ be a positive real satisfying $\mu \le \frac{1}{225}$.
Let $G$ be an $n$-vertex graph of minimum degree at least $15\sqrt{\mu}n$
with a $\mu n$-bounded incompatibility system $\mathcal{F}$.
Suppose that $P=(v_0, v_1, \ldots,v_\ell)$ is a smooth path in $G$,
where there is no vertex $w \notin V(P)$ for which $(w, v_0, \ldots, v_\ell)$
is a smooth path. Then there exists a subset $X$ of the set of good neighbors of $v_0$ in $P$,
of size at least
\[ |X| \ge d(v_0)-14\sqrt{\mu}n \]
such that for every vertex $v_{i}\in X$, the path
$(v_{i-1},\ldots,v_1,v_{0},v_{i},v_{i+1},\ldots,v_{\ell})$ is smooth.
\end{lem}
\begin{proof}
Let $V = V(G)$. Define
\begin{eqnarray*}
B_{1} & = & \Big\{ w\,:\, \{v_0, w\} \mbox{ is incompatible with } \{v_0, v_1\}, \{v_\ell, w\},  \mbox{ or } \{v_0, v_{\ell}\} \mbox{ (if exists) } \Big\}, \mbox{ and}\\
B_{2} & = & \Big\{ w\,:\, w\mbox{ is }2\sqrt{\mu}\mbox{-bad for \ensuremath{P}}\Big\} \cup \Big\{ w\,:\, w\mbox{ is }\sqrt{\mu}\mbox{-correlated with \ensuremath{v_{\ell}}}\Big\}.
\end{eqnarray*}
We have $|B_1| \le 2 \mu n + \sqrt{\mu}n$ since $\mathcal{F}$ is $\mu n$-bounded, and
the pair of vertices $v_0, v_\ell$ are $\sqrt{\mu}$-uncorrelated.
We have $|B_2| \le 2\sqrt{\mu} n$ by Proposition \ref{prop:bound_bad_correlated}.

Suppose that $v_{0}$ has a neighbor $w$ in $V\setminus(V(P)\cup B_{1}\cup B_{2})$.
We claim that the path $P'=(w,v_{0},v_{1},\cdots,v_{\ell})$ is smooth.
First, the path $P'$ is compatible since $w\notin B_{1}$ implies that
the pair of edges $\{w,v_0\}$ and $\{v_0,v_1\}$ is compatible.
Second, since $w\notin B_{2}$, we see that $w$ and $v_{\ell}$
are $\sqrt{\mu}$-uncorrelated. It remains to show that $w$
and $v_{\ell}$ are both $8\sqrt{\mu}$-good for $P'$. For $w$,
a vertex $v_{i}$ is a bad neighbor of $w$ in $P'$
if and only if it is in $P$. Hence by the fact that
$w\in B_{2}$, we see that $w$ is $2\sqrt{\mu}$-good for $P'$.
For the other endpoint $v_{\ell}$, its
set of bad neighbors in $P$ can be different from that in $P'$
in at most two vertices $v_0$ and $w$. However, since $w \in B_1$, 
$w$ cannot be a bad neighbor of $v_\ell$ in $P'$, and
$v_0$ can be a bad neighbor of $v_\ell$ in $P'$ if and only if it were
in $P$. Hence $P'$ is a smooth path, contradicting
our assumption.
This shows that all the neighbors of $v_{0}$ are in $V(P)\cup B_{1}\cup B_{2}$.

Further define
\begin{eqnarray*}
B_{3} & = & \Big\{ w \in V(P) \,:\, w \,\text{ is a bad neighbor of }v_{0}\,\mbox{ in}\, P\Big\}.
\end{eqnarray*}
We have $|B_3| \le 8\sqrt{\mu} n$, since $v_0$ is $8\sqrt{\mu}$-good for $P$.
Define $B_2^{+} = \{v_{i+1} \,|\, v_i \in B_2 \cap V(P)\}$,
and $X=\Big(N(v_{0})\cap V(P)\Big)\setminus(B_{1}\cup B_{2}^{+}\cup B_{3})$.
Since all neighbors of $v_0$ are in $V(P) \cup B_1 \cup B_2$, we have
\begin{align*}
	|X|
	\ge\,&
	|N(v_{0})|- |B_1| - 2|B_2| - |B_3| \\
	\ge\,&
	|N(v_0)| - 13\sqrt{\mu}n - 2\mu n
	\ge
	d(v_0) - 14\sqrt{\mu}n\,.
\end{align*}
We prove that
$X$ is the set claimed in the statement of the lemma.
It suffices to prove that for every $v_{i}\in X$,
the path $P''=(v_{i-1},\cdots,v_1, v_{0},v_{i},v_{i+1},\cdots,v_{\ell})$
is smooth. First, to check compatibility, we need to check the
two pairs of edges $\Big(\{v_0, v_1\}, \{v_0, v_i\}\Big)$
and $\Big(\{v_0, v_i\}, \{v_i, v_{i+1}\}\Big)$. The first pair
is compatible since $v_i \notin B_1$, and the second pair
is compatible since $v_i \notin B_3$.
Second, since $v_{i}\notin B_{2}^{+}$,
we see that $v_{i-1}\notin B_{2}$, and thus the pair of vertices $v_{i-1}$
and $v_{\ell}$ are $\sqrt{\mu}$-uncorrelated. It remains to show
that the two endpoints of $P''$ are both $8\sqrt{\mu}$-good for $P''$.

Note that for each vertex, the set of its bad neighbors
in $P''$ can be different from that in $P$ in at most two vertices,
$v_0$ and $v_i$. Since $v_{i}\notin B_{2}^{+}$,
we see that $v_{i-1}\notin B_{2}$, and thus $v_{i-1}$ has less than
$2\sqrt{\mu}|P|$ bad neighbors in $P$. This proves
that $v_{i-1}$ has less than $2\sqrt{\mu}|P|+2$ bad neighbors
in $P''$; thus $v_{i-1}$ is $8\sqrt{\mu}$-good
for $P''$  (we used the fact $|P| \ge |X| \ge \sqrt{\mu}n$
following from the given minimum degree condition).
For the other endpoint $v_{\ell}$,
consider its relation with the two vertices
$v_0$ and $v_i$. Since $v_i \notin B_1$, the pairs
of edges $\Big(\{v_{\ell},v_{i}\}, \{v_{0},v_{i}\}\Big)$
and $\Big(\{v_{\ell}, v_0\}, \{v_0, v_i\}\Big)$
are compatible (note that $\{v_\ell, v_i\}$
and $\{v_\ell, v_0\}$ may not be edges of $G$, in which
case there is no need to consider the compatibility of pairs
involving them).
Therefore, the set of bad neighbors of $v_{\ell}$
in $P''$ is a subset of that in $P$. This proves our claim
that $P''$ is smooth.
\end{proof}

\section{Proof of Theorem \ref{thm:main}} \label{sec:proof_main}

We prove Theorem \ref{thm:main} in two steps.
In the first step, we show that the given graph
contains a compatible Hamilton cycle unless it
has a special structure.

\begin{thm} \label{thm:sub_1}
Let $\mu$ be a positive real satisfying $\sqrt{\mu} < \frac{1}{400}$.
Let $G$ be an $n$-vertex graph of minimum degree at least
$\frac{n}{2}$ with a $\mu n$-bounded incompatibility system $\mathcal{F}$.
Then at least one of the following holds:
\vspace{-0.2cm}
\begin{itemize}
  \setlength{\itemsep}{0pt} \setlength{\parskip}{0pt}
  \setlength{\parsep}{0pt}
\item[(i)] There exists a Hamilton cycle compatible with $\mathcal{F}$, or
\item[(ii)] there exist two sets $A$ and $B$
of sizes $|A|,|B| \ge (\frac{1}{2}-200\sqrt{\mu})n$ such that $e(A,B) \le 16\sqrt{\mu} n^2$.
\end{itemize}
\end{thm}

In the second step, we show that there exists a
compatible Hamilton cycle even in the case when the graph
has a special structure.

\begin{thm} \label{thm:sub_2}
Let $\mu, \nu$ and $\eta$ be reals satisfying
$110\nu + 250\eta + 10\sqrt{\mu} < \frac{1}{2000}$.
Let $G$ be an $n$-vertex graph of minimum degree at least
$\frac{n}{2}$ with a $\mu n$-bounded incompatibility system $\mathcal{F}$,
and suppose that there are two sets $A,B$ of sizes
$|A|, |B| \ge (\frac{1}{2}-\nu)n$ such that $e(A,B) \le \eta n^2$.
Then $G$ contains a Hamilton cycle compatible with $\mathcal{F}$.
\end{thm}

The two theorems imply Theorem \ref{thm:main} with $\mu = 10^{-16}$.

\subsection{Step I : Theorem \ref{thm:sub_1}}

In the following lemma, we prove, by utilizing Lemma \ref{lem:rotation_1},
that every smooth path can be closed into a compatible cycle, after several rounds
of rotations and extensions.
We consider graphs that have minimum degree close to
$\frac{n}{2}$, but not necessarily at least $\frac{n}{2}$.
This extra flexibility will be useful for our later
application.

\begin{lem} \label{lem:rotation_2}
Let $\alpha$ and $\mu$ be non-negative reals satisfying
$\alpha \le \sqrt{\mu} \le \frac{1}{400}$.
Let $G$ be an $n$-vertex graph of minimum degree at least
$\left(\frac{1}{2}-\alpha\right)n$ with a $\mu n$-bounded
incompatibility system $\mathcal{F}$. Then at least one of the
following holds:
\vspace{-0.2cm}
\begin{itemize}
  \setlength{\itemsep}{0pt} \setlength{\parskip}{0pt}
  \setlength{\parsep}{0pt}
\item[(i)] For every smooth path $P$, there exists a compatible cycle $C$
of length $|C|\ge(\frac{1}{2}+20\sqrt{\mu})n$ for which $V(P) \subseteq V(C)$,
and $|E(C) \setminus E(P)| \le 3|V(C) \setminus V(P)| + 4$.
\item[(ii)] There exist two sets $A$ and $B$
of sizes $|A|,|B| \ge (\frac{1}{2}-200\sqrt{\mu})n$ such that $e(A,B) \le 15\sqrt{\mu}n^2$.
\end{itemize}
\end{lem}
\begin{proof}
Let $G$ be a given graph for which (ii) does not hold.
Given an arbitrary smooth path $P=(v_{0},v_{1},\cdots,v_{\ell})$ in $G$,
we will either extend $P$ in at most two rotations,
or show that $|P| \ge (\frac{1}{2} + 20\sqrt{\mu})n$ 
and close $P$ into a cycle in at most three rotations. If the former event happens,
then repeat the above until the latter event happens. Since we use at most three 
new edges in the former case, and four new edges in the latter case, the final
cycle $C$ that we obtain will satisfy
\[
	|E(C) \setminus E(P)| \le 3|V(C) \setminus V(P)| + 4.
\]
This implies (i). To prove our claim, it suffices to 
assume that $P$ cannot be extended in at most two rotations,
and show that under this assumption,
it can be closed into a cycle in at most three rotations.
For a set $X \subseteq V(P)$,
define $X^+ = \{v_{i+1} \,|\, v_i \in X ,\, i\le \ell-1 \}$
and $X^- = \{v_{i-1} \,|\, v_i \in X ,\, i\ge 1 \}$.

Let $S$ be the subset of vertices of $V(P)$ having the following
property: for every $v_{i}\in S$, there exists a smooth path $P'$
of length $\ell$ between $v_{i}$ and $v_{\ell}$,
which is obtained from $P$ in at most three rounds of rotations and
starts with edge $\{v_i, v_{i-1}\}$ or $\{v_i, v_{i+1}\}$.
Note that if $v_{\ell-1} \in S$, then $v_\ell$ must have been used as
a pivot point. However, since $v_\ell$ is an endpoint of the path, this implies the existence of a cycle of length $V(P)$ obtained from $P$ by adding at most four edges. Thus we may assume that $v_{\ell-1} \notin S$, which in particular implies that for all paths $P'$ as above, the edge incident to $v_{\ell}$ is still $\{v_{\ell-1}, v_{\ell}\}$.

We prove the lemma by proving that
$|S| \ge (\frac{1}{2}+20\sqrt{\mu})n$. Assume for the moment that this bound holds. 
By Lemma~\ref{lem:rotation_1}, there exists a subset of the set of good neighbors of $v_\ell$ in $P$,
of size at least $\Big(\frac{1}{2} - \alpha\Big)n - 14\sqrt{\mu}n$.
Since
\[
	\Big(\frac{1}{2}+20\sqrt{\mu}\Big)n + \Big(\frac{1}{2} - \alpha - 14\sqrt{\mu}\Big)n > n,
\]
we can find a vertex $v_i \in S$ that is a good neighbor of $v_\ell$,
and is connected to $v_\ell$ by an edge compatible with $\{v_\ell, v_{\ell-1}\}$.
Since $v_i \in S$, by definition, there exists a smooth path $P_i$ from $v_i$ to $v_\ell$
obtained from $P$ whose edge incident to $v_i$
is $\{v_i, v_{i+1}\}$ or $\{v_i, v_{i-1}\}$. This shows that
we may use the edge $\{v_\ell, v_i\}$ to close $P_i$ into a
cycle $C$ compatible with $\mathcal{F}$. Note that
$|C| \ge |S| \ge (\frac{1}{2} + 20\sqrt{\mu})n$.

Hence it suffices to prove $|S|\ge(\frac{1}{2}+20\sqrt{\mu})n$. Assume
to the contrary that $|S|<(\frac{1}{2}+20\sqrt{\mu})n$.
We show that $S$ must have some specific structure under this
assumption.

\medskip

\noindent\textbf{Claim.} Suppose that $P$ cannot be extended in at most two
rotations. If $|S| < (\frac{1}{2} + 20\sqrt{\mu})n$, then
$|S^- \cup S^+| \le \left(\frac{1}{2}+200\sqrt{\mu}\right)n$.

\medskip

The proof of this claim will be given later.
By Lemma \ref{lem:rotation_1} and our assumption on $P$
not being extendable, there exists a set $X\subseteq N(v_{0})\cap V(P)$
of size at least $d(v_0)-14\sqrt{\mu}n$ such that for
every vertex $v_{i}\in X$, the path $(v_{i-1},\ldots,v_{0},v_{i},v_{i+1},\ldots,v_{\ell})$
is a smooth path. For each vertex $v_{i}\in X$,
we similarly obtain a set
$X_i \subseteq N(v_{i-1})$ of size $|X_i| \ge d(v_{i-1}) - 14\sqrt{\mu}n$
such that each vertex $v_j \in X_i \setminus \{v_i, v_{i+1}\}$
can be used as a pivot point to give either $v_{j-1}$ or $v_{j+1}$
as another endpoint in $S$.
Hence the definition of $S$ implies that
$X_{i} \setminus \{v_i, v_{i+1}\} \subseteq S^{-} \cup S^{+}$.
Since $X_i \subseteq N(v_{i-1})$, this implies that
the number of edges between $X^-$ and $S^{-}\cup S^{+}$ satisfies
\begin{eqnarray*}
e(X^{-},S^{-}\cup S^{+}) & \ge & \sum_{x\in X^{-}}(d(x)-14\sqrt{\mu}n - 2) \ge e(X^{-},V)-14\sqrt{\mu}n\cdot|X| - 2n,
\end{eqnarray*}
and thus
\[
e\Big(X^{-},V\setminus(S^{-}\cup S^{+})\Big)\le 15\sqrt{\mu}n^{2}.
\]
However, this gives sets $A = X^-$ and $B = V \setminus (S^{-} \cup S^{+})$
satisfying (ii), thus contradicting our assumption.
\end{proof}

It remains to prove the claim. The intuition behind this somewhat
peculiar looking claim comes from the following examples of
 graphs that have minimum degree close to $\frac{n}{2}$
but are not Hamiltonian. First is the graph $G$
on $2k+1$ vertices consisting of two complete graphs $K_{k+1}$
sharing a single vertex.
There exists a Hamilton path $P$ in this graph, but it cannot be closed
into a Hamilton cycle;
note that the set $S$ as in the proof of the lemma above
consists of the first half of the path, and
we have $|S^+ \cup S^-| \approx \frac{n}{2}$.
Second is the complete bipartite graph $K_{k,k+1}$. Again, there exists
a Hamilton path $P$ that cannot be closed into a Hamilton cycle.;
the set $S$ consists of every other vertex along the path,
and we have $|S^+ \cup S^-| \approx \frac{n}{2}$. Thus
informally, our claim asserts that the given graph resembles such
graphs when $|S| \approx \frac{n}{2}$.

\begin{proof}[Proof of Claim]
Recall that $P = (v_0, \ldots, v_\ell)$ is a smooth path that cannot
be extended in at most two rotations. We defined $S$ as a
subset of vertices of $V(P)$ having the following
property: for every $v_{i}\in S$, there exists a smooth path $P'$
of length $\ell$ between $v_{i}$ and $v_{\ell}$,
which is obtained from $P$ in at most three rounds of rotations and
starts with edge $\{v_i, v_{i-1}\}$ or $\{v_i, v_{i+1}\}$.
Moreover, we assumed that $|S| \le (\frac{1}{2}+20\sqrt{\mu})n$.
For a set $X \subseteq V(P)$,
define $X^+ = \{v_{i+1} \,|\, v_i \in X ,\, i\le \ell-1 \}$
and $X^- = \{v_{i-1} \,|\, v_i \in X ,\, i\ge 1 \}$.
An {\em interval} is a set of vertices $I \subseteq V(P)$ of the form
$\{ v_j \,|\, j \in [a,b]\}$ for some $0\le a < b \le \ell$.
Throughout the proof, we sometimes add constants to inequalities,
such as in $|I \cap X| \le |I \cap X^-| + 1$ for an interval $I$
and a set $X \subseteq V(P)$,
in order to account for potential boundary effects.

\begin{figure}[b]
  \centering
  \input{fig-twicerotate}
  \caption{Obtaining the same endpoint in two different ways.}
  \label{fig_twicerotate}
\end{figure}
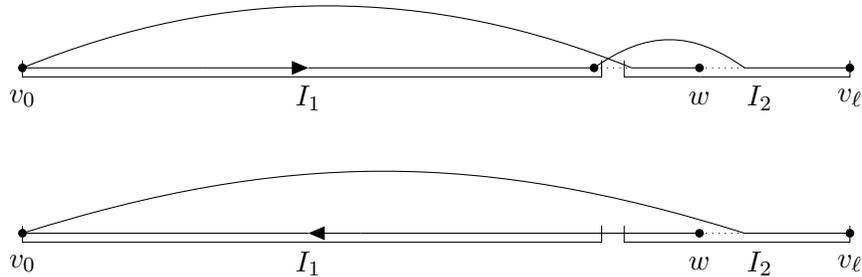

We first prove the existence of a vertex $w \in V(P)$, a `large'
interval $I \subseteq V(P)$, and two smooth paths $P_1$ and $P_2$
of length $\ell$ between $w$ and $v_{\ell}$, where $P_1$ traverses
the interval $I$ positively and $P_2$ traverses
it negatively (see Figure \ref{fig_twicerotate}).
By Lemma \ref{lem:rotation_1} and our assumption on $P$ not
being extendable, there exists
a set $X\subseteq N(v_{0})\cap V(P)$ of size at least $d(v_0) - 14\sqrt{\mu}n$
such that for every vertex $v_{i}\in X$, the path
$(v_{i-1},\ldots,v_{0},v_{i},v_{i+1},\ldots,v_{\ell})$ is smooth.
By definition of $S$, we thus have $X^{-}\subseteq S$.
Let $\beta = 70\sqrt{\mu} + 3\alpha$, and take the
vertex $v_{j}\in X$ whose index $j$ is $\beta n$-th largest among
those in $X$.
Such index exists since
$|X| \ge \delta(G) - 14\sqrt{\mu}n \ge \beta n$,
by the given condition on parameters.

Let $I_{1}=\{v_{0},v_{1},\ldots,v_{j-1}\}$ and $I_{2}=\{v_{j},\ldots,v_{\ell}\}$.
Note that
\begin{align} \label{eq:interval_bound}
	|I_1 \cap S| = |S| - |I_2 \cap S| \le |S| - |I_2 \cap X^-| \le \Big(\frac{1}{2} + 20\sqrt{\mu} - \beta \Big)n + 1.
\end{align}
Let $P'=(v_{j-1},v_{j-2},\cdots,v_1,v_{0},v_{j},v_{j+1},\cdots,v_{\ell})$.
Note that $P'$ cannot be extended 
since $P$ cannot be extended in at most two rotations.
Thus by applying Lemma \ref{lem:rotation_1} to $P'$, we see, as above, that there
exists a set $Y\subseteq N(v_{j-1}) \setminus \{v_j, v_{j+1}\}$ of size
$|Y| \ge d(v_{j-1}) - 14\sqrt{\mu}n - 2$ such that $I_1 \cap Y^+ \subseteq S$
and $I_2 \cap Y^- \subseteq S$.
If there exists a vertex $w \in I_2 \cap X^{-} \cap Y^{-}$, then
there exist two paths $P_+$, $P_-$ and an interval $I=I_1$ as claimed
(see Figure \ref{fig_twicerotate}).
Assume for the contrary that $I_2 \cap X^{-} \cap Y^{-} = \emptyset$.
Then since $I_2 \cap X^{-}$ and $I_2 \cap Y^{-}$ are both subsets of $I_2 \cap S$,
we see that
\begin{align*}
	|X|+|Y|
	=\,& |I_1 \cap X|+|I_2 \cap X|+|I_1 \cap Y|+|I_2 \cap Y| \\
	\le\,& |I_1 \cap X^-|+|I_2 \cap X^-|+|I_1 \cap Y^+|+|I_2 \cap Y^-| + 4 \\ 
	\le\,& 2 \cdot |I_1 \cap S|+|I_2 \cap S| + 4,
\end{align*}
where we used the fact that $I_1 \cap X^-$, $I_1 \cap Y^+$, $I_2 \cap X^-$,
and $I_2 \cap Y^-$ are all subsets of $S$, and that $I_2 \cap X^-$
and $I_2 \cap Y^-$ are disjoint. Since $|I_1 \cap S| + |I_2 \cap S| = |S|$,
we see by \eqref{eq:interval_bound} that
\[
	|X|+|Y| \le |I_1 \cap S| + |S| + 4 \le (1 + 40\sqrt{\mu} - \beta)n + 5.
\]
which is a contradiction, since
\[
	\min\{ |X|, |Y|\} \ge \delta(G) - 14\sqrt{\mu}n - 2 \ge \Big(\frac{1}{2} - \alpha - 14\sqrt{\mu}\Big)n - 2,
\]
and $\beta = 70\sqrt{\mu} + 3\alpha$.

Hence we proved the existence of a vertex $w \in V(P)$, an
interval $I_1 \subseteq V(P)$, and two smooth paths $P_1$ and $P_2$ of length $\ell$
between $w$ and $v_{\ell}$, where $P_1$ traverses
the interval $I_1$ positively and $P_2$ traverses it negatively.
Note that if a vertex $v_i \in N(w) \cap I_1$ for $0 < i < j-1$
is used as a pivot point in $P_1$, then we obtain $v_{i-1}$ as a new endpoint, and
if used in $P_2$, then we obtain $v_{i+1}$ as a new endpoint.
Since $0 < i < j - 1$, then 
the edge incident to the new endpoint belongs to the original path $P$
for both cases.

Since $P$ cannot be extended in at most two rotations, 
by Lemma \ref{lem:rotation_1}, there exists a set
$Z_1 \subseteq N(w) \cap I_1$ of size at least $|N(w) \cap I_1| - 14\sqrt{\mu}n$ that can be used as pivot points of $P_1$, and a set $Z_2 \subseteq N(w) \cap I_1$ of
size at least $|N(w) \cap I_1| - 14\sqrt{\mu}n$ that can be used as pivot
points of $P_2$. Therefore, the set $Z = Z_1 \cap Z_2$ can be used as pivot
points for both paths, and has size
\[ 
	|Z| \ge |N(w) \cap I_1| - 28\sqrt{\mu}n. 
\]
Moreover, the observation above shows that $Z \setminus \{v_0, v_{j-1}\} \subseteq S^{-}\cap S^{+}$.
Therefore
\begin{align*}
	|S^{-}\cap S^{+}| \ge |Z| - 2
	\ge\,& |N(w) \cap I_1|-28\sqrt{\mu}n - 2.
\end{align*}
Let $Z_w \subseteq N(w) \cap V(P)$ be the set obtained by applying Lemma~\ref{lem:rotation_1}
to the path $P_2$. Since $Z_w \subseteq N(w)$ and $|Z_w| \ge (\frac{1}{2} - \alpha)n - 14\sqrt{\mu}n$,
\begin{align*}
	|S^- \cap S^+| &\ge |Z_w \cap I_1|-28\sqrt{\mu}n - 2 = |Z_w|  - |Z_w \cap I_2| - 28\sqrt{\mu}n - 2\\
	&\ge \Big(\frac{1}{2} - \alpha\Big)n - 14\sqrt{\mu}n - |Z_w \cap I_2| - 28\sqrt{\mu}n - 2.
\end{align*}
Since all vertices in $(Z_w \cap I_2) \setminus \{v_j, w^+, w^{++}\}$ can be used as pivot points for path $P_2$ to give new endpoints in $I_2 \cap S$, we see that
$|Z_w \cap I_2| \le |I_2 \cap S| + 3$. Therefore,
\begin{align*}
	|S^- \cap S^+| \ge\,& \Big(\frac{1}{2} - \alpha\Big)n-|I_2 \cap S| - 42\sqrt{\mu} n - 5.
\end{align*}
Since
\begin{align*}
	|I_2 \cap S| &\le |I_2 \cap X^{-}| + |S \setminus X^{-}| = (\beta n - 1) + (|S| - |X^{-}|) \\
	&\le
	\Big(\frac{1}{2} + 20\sqrt{\mu} + \beta \Big)n - \Big(\frac{1}{2} - \alpha - 14\sqrt{\mu}\Big)n
	\le
	(\alpha + \beta + 34\sqrt{\mu})n,
\end{align*}
we see that
\[
	|S^{-} \cap S^{+}| \ge \left(\frac{1}{2} - 2\alpha - \beta - 77\sqrt{\mu} \right)n.
\]
From this, since $\beta = 70\sqrt{\mu} + 3\alpha$, we obtain
\begin{eqnarray*}
|S^{-}\cup S^{+}| = |S^{-}|+|S^{+}|-|S^{-}\cap S^{+}|
  &\le&  2\left(\frac{1}{2}+20\sqrt{\mu}\right)n - \left(\frac{1}{2} - 2\alpha - \beta - 77\sqrt{\mu} \right)n \\
 & = & \left(\frac{1}{2}+5\alpha+187\sqrt{\mu}\right)n \le \left(\frac{1}{2}+200\sqrt{\mu}\right)n.
\end{eqnarray*}
\end{proof}

By using Lemma \ref{lem:rotation_2}, we can now prove Theorem \ref{thm:sub_1}.

\begin{proof}[Proof of Theorem \ref{thm:sub_1}]
Remove from $G$ all the edges whose two
endpoints are $\sqrt{\mu}$-correlated.
By Proposition \ref{prop:bound_bad_correlated}, we know that
the resulting graph $G'$ has minimum degree at least
$\Big(\frac{1}{2} - \sqrt{\mu}\Big)n$.
If there exist two sets $A$ and $B$ of sizes
$|A|,|B| \ge (\frac{1}{2}-200\sqrt{\mu})n$ such that
$e_{G'}(A,B)\le 15\sqrt{\mu}n^2$, then we have
$e_{G}(A,B) \le 16\sqrt{\mu}n^2$, and thus
alternative (ii) holds. Hence we may assume that
there are no such two sets $A$ and $B$.
From now on, we will only consider the graph $G'$. Thus by abusing
notation, we let $G$ be the graph $G'$. Note that Lemma \ref{lem:rotation_2}~(i) applies to our graph.

Let $C = (v_0, \ldots, v_\ell, v_0)$ be a
maximum length compatible cycle in $G$.
Throughout the proof, for a set $X \subset V$,
define $X^+ = \{ v_{i+1} \,:\, v_i \in X \cap V(C)\}$, and
$X^- =  \{ v_{i-1} \,:\, v_i \in X \cap V(C)\}$ (where index
addition and subtraction are modulo $\ell+1$).
By Lemma \ref{lem:rotation_2},
we have $|C| \ge (\frac{1}{2} + 20\sqrt{\mu})n$.
If $C$ is a Hamilton cycle, then
we are done. Otherwise, there exists a vertex $z$ not
in the cycle.

Define $B_0$ as the set of vertices that are $2\sqrt{\mu}$-bad
for $C$. By Proposition \ref{prop:bound_bad_correlated},
we know that $|B_0| \le \sqrt{\mu}n$. Hence if $|C| < (1-\sqrt{\mu})n$,
then we may take $z$ to be a vertex not in $B_0$. In this case,
define $B_1$ as the set of bad neighbors of $z$ in $C$. By definition,
we have $|B_1| \le 2\sqrt{\mu}|C|$.
Otherwise if $|C| \ge (1-\sqrt{\mu})n$, then let $z$ be an
arbitrary vertex not in $C$, and define $B_1 = \emptyset$.
Since
\[
	|N(z) \cap V(C)|
	\ge d(z) + |V(C)| - n
	\ge \Big(\frac{1}{2} - \sqrt{\mu}\Big)n + \Big(\frac{1}{2} + 20\sqrt{\mu}\Big)n - n
	\ge 7\sqrt{\mu}n
\]
and
\[
	3|B_0| + |B_1| \le 5\sqrt{\mu}n,
\]
the set $T = \{ v_i \in V(C) \,|\, v_i \in N(z), v_{i-1}, v_{i+1} \notin B_0, v_i \notin B_0 \cup B_1 \}$
has cardinality
\[
	|T| \ge 2\sqrt{\mu}n.
\]

Take a vertex $v_i \in T$. Since $\mathcal{F}$ is $\mu n$-bounded,
there are at most $\mu n$ vertices $x \in T$ for which
the edges $\{z, v_i\}$ and $\{z, x\}$ are incompatible.
Also, by Proposition \ref{prop:bound_bad_correlated}, there are
at most $\sqrt{\mu}n$ vertices $x \in T$ for which $x^+$
is $\sqrt{\mu}$-uncorrelated with $v_{i+1}$. 
Therefore, since $|T| \ge 2\sqrt{\mu}n$, we can find a vertex $v_j \in T$ for which the edges
$\{z, v_i\}$ and $\{z, v_j\}$ are compatible, the
pair of vertices $v_{i+1}$ and $v_{j+1}$ is $\sqrt{\mu}$-uncorrelated,
and $v_j \neq v_i, v_{i-1}, v_{i+1}$.
Consider the path
$P = (v_{i+1}, \ldots, v_j, z, v_i, v_{i-1}, \ldots, v_{j+1})$.
First, the two endpoints of $P$ are
$\sqrt{\mu}$-uncorrelated by the choice of $v_i$ and $v_j$,
and second, both $v_{i+1}$ and $v_{j+1}$ have at most
$2\sqrt{\mu}|C| + 3$ bad neighbors in $P$, since
$v_{i+1}, v_{j+1} \notin B_0$ and the set of bad neighbors
in $P$ and in $C$ can differ only in at most three vertices
$v_j, z$, and $v_i$.

To check whether $P$ is compatible, it suffices to check the
compatibility of the three pairs
$\Big(\{v_{j-1}, v_j\}, \{v_j, z\}\Big)$,
$\Big(\{v_j, z\}, \{z, v_i\}\Big)$, and
$\Big(\{z, v_i\}, \{v_i, v_{i-1}\}\Big)$. The pair of edges
$\{v_j, z\}$ and $\{z, v_i\}$ is compatible by our choice
of $v_i$ and $v_j$.
If $|C| < (1-\sqrt{\mu})n$, then by the choice of $z$ and
the set $B_1$, since $v_i, v_j \notin B_1$ (this follows from 
$v_i, v_j \in T$), we further see
that the two other pairs are both compatible.
Hence $P$ is compatible, and therefore smooth.
By Lemma \ref{lem:rotation_2},
this gives a compatible cycle longer than $C$, contradicting
the maximality of $C$.

Therefore, we must have $|C| \ge (1-\sqrt{\mu})n$.
In this case, $P$ is `almost' smoothly compatible with $\mathcal{F}$,
in the sense that it satisfies all the conditions except for possibly the
compatibility of two pairs of edges.
Define $\mathcal{F}_1$ as the incompatibility system obtained from
$\mathcal{F}$ by making the pairs of edges
$\Big(\{v_{j-1}, v_j\}, \{v_j, z\}\Big)$ and
$\Big(\{z, v_i\}, \{v_i, v_{i-1}\}\Big)$ to be compatible.
Note that $P$ is smoothly compatible with $\mathcal{F}_1$.
Hence by Lemma \ref{lem:rotation_2}, we can find a cycle
$C_1$ compatible with $\mathcal{F}_1$, with
$V(C_1) \supseteq V(P) \supseteq V(C)$ and 
\[
	|E(C_1) \setminus E(C)|
	\le |E(C_1) \setminus E(P)| + 2
	\le 3|V(C_1) \setminus V(C)| + 6.
\]

Let $P_1$ be the path obtained from $C_1$ by removing the
edge $\{v_{j-1}, v_j\}$ if it is in $C_1$ (if it is not in
$P_1$, then skip the rest of this paragraph).
We claim that $P_1$ is smoothly compatible with $\mathcal{F}_1$.
First, it is compatible with $\mathcal{F}_1$, since $C_1$ is.
Second, the two endpoints are $\sqrt{\mu}$-uncorrelated,
since we started by removing all edges whose two
endpoints are $\sqrt{\mu}$-correlated.
Third, since $v_j \notin B_0$, we know that
$v_j$ is $2\sqrt{\mu}$-good for $C$. Since $V(C_1) \supseteq V(C)$, it
follows that $v_j$ has at most
\[
	2\sqrt{\mu}|C| + 2|E(C_1) \setminus E(C)| \le 2\sqrt{\mu}|C| + 6(|C_1|-|C|) + 12
	\le
	8\sqrt{\mu} |C_1| = 8\sqrt{\mu}|P_1|
\]
bad neighbors in $P_1$, where the final inequality follows from
$|C| \ge (1-\sqrt{\mu})n$. A similar estimate holds for the
other endpoint $v_{j-1}$.
Let $\mathcal{F}_2$ be the incompatibility system obtained from
$\mathcal{F}$ by making only the pair
$\Big(\{z, v_i\}, \{v_i, v_{i-1}\}\Big)$ to be compatible.
Note that $P_1$ is smoothly compatible with $\mathcal{F}_2$ as well,
since it does not contain the edge $\{v_{j-1}, v_j\}$.
Hence we can find, by Lemma \ref{lem:rotation_2}, a cycle
$C_2$ compatible with $\mathcal{F}_2$.

Repeat the argument above for $C_2$ and $\mathcal{F}_2$,
to find a path $P_2$ that is smoothly compatible with $\mathcal{F}_2$,
not containing the edge $\{v_i, v_{i-1}\}$. This path
is smoothly compatible with $\mathcal{F}$, and thus by Lemma \ref{lem:rotation_2},
we can find a cycle compatible with $\mathcal{F}$ whose vertex set contains $V(C) \cup \{z\}$,
contradicting the maximality of $C$. Therefore, the given graph contains a Hamilton cycle compatible with $\mathcal{F}$.
\end{proof}

\subsection{Step II : Theorem \ref{thm:sub_2}}

In this subsection, we consider the case when $G$ contains
two large subsets $A$ and $B$ with few edges between them.
We first show that in this case $A$ and $B$ are either almost disjoint,
or almost identical. Afterwards, for each case, we further
process the graph to convert the problem into a problem
of establishing `compatible Hamilton connectivity' of almost
complete graphs, and a problem of establishing `compatible
Hamiltonicity' of almost bipartite graphs. The proof of
these final pieces are very similar in structure to the proof
of Theorem \ref{thm:sub_1}, but are different in detail,
and will be given in the following section.

\medskip

Suppose that $G$ is an $n$-vertex graph of minimum degree at least $\frac{n}{2}$,
and let $V = V(G)$. Suppose that there exist two sets $A$ and $B$ of sizes
$|A|, |B| \ge (\frac{1}{2} - \nu)n$ such that $e(A,B) \le \eta n^2$.
We have
\begin{align*}
	e(A,B)
	\ge\,& e(A\cap B, A \cup B) \\
	\ge\,& |A \cap B| \cdot (\delta(G) - (n - |A \cup B|)) \\
	=\,& |A \cap B| \cdot (\delta(G) + |A| + |B| - |A \cap B|- n).
\end{align*}
The cardinalities of $A,B$, the bound $e(A,B) \le \eta n^2$, and the
bound $\delta(G) \ge \frac{n}{2}$
imply that
\begin{align*}
	|A \cap B| \cdot \Big(\frac{n}{2} - 2\nu n - |A \cap B|\Big) \le \eta n^2.
\end{align*}
If $3\eta n \le |A \cap B| \le (\frac{1}{2} - 2\nu - 3\eta)n$, then
the left-hand-side above is greater than $\eta n^2$, since
$3\eta \cdot (\frac{1}{2} - 2\nu - 3\eta) > \eta$. Hence, we must have
\[
	|A \cap B| < 3\eta n \quad \text{ or } \quad |A \cap B| > \Big(\frac{1}{2} - 2\nu - 3\eta\Big)n.
\]
We consider the two cases separately.

\medskip

\noindent {\bf Case 1}. $|A \cap B| < 3\eta n$.

If $|A \cap B| < 3\eta n$, then there exist disjoint sets
$A' \subseteq A$ and $B' \subseteq B$ satisfying
\[
	|A'|, |B'|
	\ge
	\Big(\frac{1}{2} - \nu\Big)n - \frac{|A\cap B|}{2}
	\ge
	\Big(\frac{1}{2} - \nu - \frac{3}{2}\eta\Big)n,
\]
and $e(A', B') \le e(A,B) \le \eta n^2$. Furthermore, by considering
a random partition of the vertices not in $A' \cup B'$,
we can obtain a partition $A'' \cup B''$ of the vertex set so that
\begin{align*}
	e(A'', B'')
	\le\,& e(A', B') + \frac{1}{2} e(V \setminus (A'\cup B'), V)  \\
	\le\,& \eta n^2 + \frac{1}{2} \cdot (2\nu + 3\eta)n \cdot n
	= \Big(\nu + \frac{5}{2}\eta \Big)n^2.
\end{align*}
Consider the partition $A'' \cup B''$, and repeatedly move vertices
that have at most $\frac{n}{6}$
neighbors in their own part to the other part. Since $G$ has minimum
degree at least $\frac{n}{2}$, such vertex has at least $\frac{n}{3}$ neighbors
in the other part prior to moving, and thus each time we move a
vertex, the number of edges across the partition decreases by at least
$\frac{n}{6}$. Hence the process ends in at most $(6\nu + 15\eta)n$
steps, producing a partition $W \cup W^c$. In this partition,
both parts have size
between $(\frac{1}{2} - 7\nu - \frac{33}{2}\eta)n \ge \frac{n}{3}$ and
$(\frac{1}{2} + 7\nu + \frac{33}{2}\eta)n \le \frac{2n}{3}$,
and we have $e(W,W^c) \le e(A'', B'') \le (\nu + \frac{5}{2}\eta)n^2$.
Moreover, the minimum degrees of $G[W]$ and $G[W^c]$ are
both at least
$\frac{n}{6} \ge \frac{1}{5} \max\{|W|, |W^c|\}$,
since $\max\{|W|, |W^c|\} \le \frac{2}{3}n$.

Without loss of generality, assume that $|W| \le \frac{n}{2}$.
Note that $e(W)$ is at least
\[
	\frac{1}{2}\Big(|W|\cdot \frac{n}{2} - e(W,W^c)\Big)
	\ge
	{|W| \choose 2} - \Big(\frac{1}{2}\nu + \frac{5}{4}\eta \Big)n^2
	\ge
	{|W| \choose 2} - \Big(\frac{9}{2}\nu + \frac{45}{4}\eta \Big)|W|^2,
\]
where the last inequality follows from the bound $|W| \ge \frac{n}{3}$.
Also, $e(W^c)$ is at least
\begin{align*}
	\frac{1}{2}\Big(|W^c|\cdot \frac{n}{2} - e(W,W^c)\Big)
	\ge\,&
	{|W^c| \choose 2} - \frac{1}{2} |W^c| \Big(|W^c| - \frac{n}{2}\Big)- \Big(\frac{1}{2}\nu +
\frac{5}{4}\eta \Big)n^2 \\
	\ge\,&
	{|W^c| \choose 2} - \Big(\frac{7}{2}\nu + \frac{33}{4}\eta\Big)n|W^c| - \Big(\frac{1}{2}\nu + \frac{5}{4}\eta \Big)n^2 \\
	\ge\, &
	{|W^c| \choose 2} -  \Big(9\nu + 22\eta \Big)|W^c|^2,
\end{align*}
where the last inequality follows from the bound $|W^c| \ge \frac{n}{2}$.
Therefore, in the end, we obtain a partition $W \cup W^c$ with
the following properties:
\begin{itemize}
  \setlength{\itemsep}{0pt} \setlength{\parskip}{0pt}
  \setlength{\parsep}{0pt}
\item $\frac{n}{3} \le |W| \le \frac{n}{2}$,
\item For $X = W$ and $W^c$, the graph $G[X]$ has minimum degree at least $\frac{|X|}{5}$, and
\item For $X = W$ and $W^c$, we have $e(X) \ge {|X| \choose 2} - \Big(9\nu + 22\eta \Big)|X|^2$.
\end{itemize}

We find a Hamilton cycle compatible with $\mathcal{F}$ by first
finding two vertex disjoint edges
$e_1 =\{x_1, y_1\}$ and $e_2 = \{x_2, y_2\}$ such that $x_1, x_2 \in W$ and
$y_1, y_2 \in W^c$, and then finding a Hamilton path in $G[W]$ whose two endpoints
are $x_1$ and $x_2$, and a Hamilton path in $G[W^c]$ whose two endpoints
are $y_1$ and $y_2$; of course we need to ensure the compatibility of the so obtained cycle.

To find two vertex disjoint edges, if $|W| < \frac{n}{2}$, the minimum
degree condition of $G$ implies that each vertex in $W$ has at least
$\lceil\frac{n}{2}\rceil - (|W|-1) \ge 2$
 neighbors in $W^c$. Thus in this case, we can easily find two
vertex disjoint edges. Otherwise, if $|W| = |W^c| = \frac{n}{2}$,
then the bipartite graph induced by the partition $W \cup W^c$
has minimum degree at least $1$, which for $|W| \ge 2$
implies that the minimum vertex cover is of size at least $2$.
Therefore we can find two vertex disjoint edges in this case as well.
Let $e_1 = \{x_1, y_1\}$ and $e_2 = \{x_2, y_2\}$ be the vertex disjoint
edges that we have found, where $x_1, x_2 \in W$ and
$y_1, y_2 \in W^c$.

Let $G_1 = G[W]$, and add the edge $\{x_1, x_2\}$ if it is not already in $G_1$.
Define an incompatibility system $\mathcal{F}_1$ over $G_1$ as follows.
For two edges $e_1, e_2 \in E(G_1)$ both different from $\{x_1, x_2\}$,
let $e_1$ and $e_2$ be incompatible in $\mathcal{F}_1$ if and only if
they are incompatible in $\mathcal{F}$. For an edge $e \neq \{x_1, x_2\}$
containing $x_1$,
the  edges $\{x_1, x_2\}$ and $e$ are incompatible in $\mathcal{F}_1$
if and only if $\{x_1, y_1\}$ and $e$ are incompatible in $\mathcal{F}$.
For an edge $e \neq \{x_1, x_2\}$ containing $x_2$,
the  edges $\{x_1, x_2\}$ and $e$ are incompatible in $\mathcal{F}_1$
if and only if $\{x_2, y_2\}$ and $e$ are incompatible in $\mathcal{F}$.
One can easily check that $\mathcal{F}_1$ is a $\mu n$-bounded incompatibility system.
Similarly let $G_2 = G[W^c]$ (with the edge $\{y_1, y_2\}$ added)
and define a $\mu n$-bounded incompatibility system $\mathcal{F}_2$
over $G_2$.

Suppose that we find a Hamilton cycle $C_1$ in $G_1$ containing $\{x_1, x_2\}$ and
compatible with $\mathcal{F}_1$, and $C_2$ in $G_2$ containing $\{y_1, y_2\}$ and
compatible with $\mathcal{F}_2$. The two cycles $C_1$, $C_2$ together with the
two edges $e_1$ and $e_2$ give a Hamilton cycle $C$ in $G$ that is compatible with
$\mathcal{F}$, due to the way we defined the incompatibility systems $\mathcal{F}_1$
and $\mathcal{F}_2$. Therefore, it suffices to prove the following theorem
(we will apply it with 
$\beta_{\ref{thm:almost_clique}} = 9\nu + 22\eta$
and $\mu_{\ref{thm:almost_clique}} = 3\mu$, where the subscripts indicate
that the constants will be applied to Theorem \ref{thm:almost_clique}. The factor of $3$ in $\mu_{\ref{thm:almost_clique}}$ has been introduced since
$W$ can be as small as $\frac{n}{3}$). 
Its proof is similar to that of Theorem \ref{thm:sub_1} (and is in fact much simpler),
and will be given separately in the following section.

\begin{thm} \label{thm:almost_clique}
Let $\beta$ and $\mu$ be positive reals satisfying
$\beta + 2\sqrt{\mu} \le \frac{1}{1200}$.
Let $G$ be an $n$-vertex graph with minimum degree at least $\frac{n}{5}$ and
at least ${n \choose 2} - \beta n^2$ edges, and $\mathcal{F}$
be a $\mu n$-bounded incompatibility system over $G$.
Then for every edge $e$ of $G$, there exists a Hamilton cycle
containing $e$ that is compatible with $\mathcal{F}$.
\end{thm}

\medskip

\noindent {\bf Case 2}. $|A \cap B| \ge \Big(\frac{1}{2} - 2\nu - 3\eta\Big)n$.

Let $A' = A \cap B$ and $B' = V \setminus A'$. Note that
$e(A', A') \le e(A, B) \le \eta n^2$. Therefore,
\[ e(A', B') \ge \delta(G) \cdot |A'| - e(A', A') \ge \frac{n}{2} \cdot \Big(\frac{1}{2} - 2\nu - 3\eta\Big)n - \eta n^2
= \Big(\frac{1}{4} - \nu - \frac{5}{2}\eta\Big)n^2. \]
Repeatedly move vertices having at most $\frac{n}{6}$ neighbors
across the partition to the other part. Since $G$ has minimum degree
at least $\frac{n}{2}$, such vertex has at least $\frac{n}{3}$
neighbors in its own part prior to moving. Hence each time we move a
vertex, the number of edges across the partition increases by at least
$\frac{n}{6}$. Since the maximum possible number of edges across
a partition is $\frac{1}{4}n^2$, the process ends
in at most $(6\nu + 15\eta)n$ steps,
producing a partition $W \cup W^c$. Both parts have size
between $|A \cap B| - (6\nu + 15\eta)n \ge (\frac{1}{2} - 8\nu - 18\eta)n$
and $(\frac{1}{2} + 8\nu + 18\eta)n$,
and satisfy $e(W,W^c) \ge e(A', B') \ge \Big(\frac{1}{4} - \nu - \frac{5}{2}\eta\Big)n^2$.
Moreover, each vertex has at least $\frac{n}{6}$
neighbors across the partition.
Without loss of generality, assume that $|W| \ge \frac{n}{2}$.
Then, while the bound $|W| > \lceil \frac{n}{2} \rceil$ holds,
repeatedly move vertices $w \in W$ having at least $\frac{n}{16}$
neighbors in $W$, to the other part. Note that we move
at most $(8\nu + 18\eta)n$ vertices during this process.
In the end, we obtain a partition with the following properties.
\begin{itemize}
  \setlength{\itemsep}{0pt} \setlength{\parskip}{0pt}
  \setlength{\parsep}{0pt}
\item the bipartite graph induced by $G$ on $W \cup W^c$ has minimum degree at
least $\Big(\frac{1}{16} - 8\nu - 18\eta\Big)n$,
\item $e(W, W^c) \ge \Big(\frac{1}{4} - \nu - \frac{5}{2}\eta\Big)n^2 - \Big(8\nu + 18\eta \Big)n^2 \ge \Big(\frac{1}{4} - 9\nu - 21\eta\Big)n^2$,
\item $\lceil \frac{n}{2} \rceil \le |W| \le \Big(\frac{1}{2} + 8\nu + 18\eta\Big)n $, and
\item if $|W| > \lceil \frac{n}{2} \rceil$, then $G[W]$ has maximum degree less than $\frac{n}{16}$.
\end{itemize}

Let $|W| = \frac{n+t}{2}$ and $|W^c| = \frac{n-t}{2}$, for a
non-negative integer $t \le (16\nu + 36\eta)n$.
The Hamilton cycle that we find will make use of exactly $t$ edges
within the set $W$, and all other edges will be between $W$ and $W^c$.
We thus must first find $t$ edges within $W$ (we may assume $t > 0$).
In this case,
since $G$ has minimum degree at least $\frac{n}{2}$, we see that each
vertex $w \in W$ has at least $\frac{t}{2}$ neighbors in $W$.
Therefore, $e(W) \ge \frac{1}{2} \cdot \frac{t}{2}|W| \ge \frac{nt}{8}$.
If $t=1$, then let $e_1$ be an arbitrary edge whose
both endpoints are in $W$. Otherwise, if $t > 1$, then
since $G[W]$ has maximum degree less than $\frac{n}{16}$,
it has covering number greater than $2t$, and therefore contains
$t$ disjoint edges $e_1, \ldots, e_t$.
Let $E_0 = \{e_1, \ldots, e_t\}$. If $t=0$, then let $E_0 = \emptyset$.

Let $W_1 \subseteq W$ be a set of size $|W_1| = \frac{n-t}{2}$
that intersects each edge $e_i \in E_0$ in exactly one vertex,
and let $W_2 = W^c$. Note that $|W_1| = |W_2|$.
Our next step towards establishing Hamiltonicity is to find a perfect matching
between $W_1$ and $W_2$. This perfect matching will later
play an important role in finding a Hamilton cycle.
Consider a bipartite subgraph $H$ of $G$ obtained by the following process.
First, take only the edges between the two sets $W_1$ and $W_2$.
Then, for each $e_i \in E_0$ and its endpoint $v_i \in W_1$, remove all
edges incident to $v_i$ that are incompatible with $e_i$.
We claim that $H$ satisfies Hall's condition. Note that $H$ has minimum
degree at least
\[
	\delta(H)
	\ge
	\Big(\frac{1}{16} - 8\nu - 18\eta\Big)n - t - \mu n
	\ge
	\Big(\frac{1}{16} - 24\nu - 54\eta - \mu\Big)n
	\ge
	\frac{n}{20}.
\]
Furthermore, since the complement of $G$ has at most $(9\nu + 21\eta)n^2$ edges bewteen $W$ and $W^c$ and $H$ is obtained from an induced subgraph of $G$ by removing at most $\mu n |W_1| \le \mu n^2$ edges, 
the complement of $H$ has at most $(9\nu + 21\eta)n^2 + \mu n^2$ edges between $W_1$ and $W_2$. Therefore,
\begin{align} \label{eq:h_edges}
	e(H)
	\ge
	|W_1||W_2| - (9\nu + 21\eta)n^2 - \mu n^2
	>
	|W_1||W_2| - \frac{1}{400}n^2.
\end{align}
By the minimum degree condition of $H$, it suffices to
consider the expansion of sets $W_1' \subseteq W_1$ of size
\[
	\frac{n}{20} \le |W_1'| \le \frac{n-t}{2} - \frac{n}{20}.
\]
If a set $W_1'$ does not expand, then there exists a set $W_2'$ of size
$|W_2'| \ge \frac{n-t}{2} - |W_1|$, where
\[
	e_H(W_1', W_2') = 0.
\]
However, this implies that
$e(H) \le |W_1||W_2| - \frac{1}{400}n^2$,
contradicting the lower bound \eqref{eq:h_edges} on $e(H)$. Hence, $H$ satisfies
Hall's condition, and thus contains a perfect matching.

The perfect matching of $H$ together with the $t$ edges of $E_0$ gives
$t$ paths $(x_i, v_i, w_i)$ of length $2$, and $\frac{n-3t}{2}$ edges
$(v_j, w_j)$, all vertex-disjoint,
thus covering all vertices of the graph $G$ exactly once. Moreover, each
path of length $2$ consists of a pair of edges that are compatible.
Consider the graph $H'$ and incompatibility system $\mathcal{F}'$
obtained from $G$ and $\mathcal{F}$ by the following process.
First, consider only the edges between $W$ and $W^c$.
Then, for each index $i$ with $1\le i\le t$,
remove the vertex $v_i$, add the edge $\{x_i, w_i\}$ (if it
was not an edge in $G$), and make $\{x_i, w_i\}$ incompatible
with the edges in $G$ that are incident to $x_i$ and incompatible
with $\{x_i, v_i\}$, and are incident to $w_i$ and incompatible
with $\{v_i, w_i\}$. One can easily check that $H'$ is a balanced bipartite
graph with bipartition $\{x_1, \ldots, x_t, v_{t+1}, \ldots, v_m\} \cup \{w_1, \ldots, w_m\}$,
where $m = \frac{n-t}{2}$, and
that $\mathcal{F}'$ is a $\mu n$-bounded incompatibility system.
Moreover, $H'$ has minimum degree at least
\[
	\delta(H') \ge \Big(\frac{1}{16} - 8\nu - 18\eta\Big)n - t \ge \frac{n}{20},
\]
and its number of edges is at least
\[
	e(H')
	\ge
	\Big(\frac{n-t}{2}\Big)^2  - \Big(9\nu + 21\eta\Big)n^2
	\ge
	\Big(1 - (45\nu + 105\eta)\Big)\Big(\frac{n-t}{2}\Big)^2.
\]

Suppose that we find a Hamilton cycle $C'$ in $H'$
containing the edges $\{x_1, w_1\}, \ldots, \{x_t, w_t\}$,
that is compatible with $\mathcal{F}'$.
Consider the cycle $C$ obtained from $C'$ by replacing
each edge $\{x_i, w_i\}$ with the path $(x_i, v_i, w_i)$. Note
that $C$ is a Hamilton cycle in $G$, and is compatible
with $\mathcal{F}$ by our definition of $\mathcal{F}'$.
Therefore, this case can be settled through the following theorem,
whose proof will be given in the following section
(we will apply it with 
$\gamma_{\ref{thm:almost_bipartite}} = 64\nu + 144\eta$,
$\beta_{\ref{thm:almost_bipartite}} = 45\nu + 105\eta$,
and $\mu_{\ref{thm:almost_bipartite}} = 3\mu$, where the subscripts
indicate that the constants are being applied to Theorem \ref{thm:almost_bipartite}.  The value of $\mu_{\ref{thm:almost_bipartite}}$ is chosen as above to ensure that $\mu_{\ref{thm:almost_bipartite}} m = \mu_{\ref{thm:almost_bipartite}} \frac{n-t}{2} \ge \mu n$).

\begin{thm} \label{thm:almost_bipartite}
Let $\mu, \beta, \gamma$ be positive reals satisfying
$\gamma + \beta + \sqrt{\mu} < \frac{1}{2000}$.
Let $G$ be a bipartite graph with bipartition $A \cup B$
and minimum degree at least $\frac{m}{10}$, where $|A|=|B|=m$
and $e(A,B) \ge (1 - \beta)m^2$.
Let $\mathcal{F}$ be a $\mu m$-incompatibility
system over $G$. Further suppose that there is a
perfect matching consisting of edges $e_1, \ldots, e_m$.
Then there exists a Hamilton cycle
containing the edges $e_1, \ldots, e_{\gamma m}$,
that is compatible with $\mathcal{F}$.
\end{thm}

\section{Extremal cases}
\label{sec:proof_sub}

In this section, we prove Theorems \ref{thm:almost_clique} and
\ref{thm:almost_bipartite}, thereby completing the proof of
Theorem \ref{thm:sub_2}, and thus Theorem \ref{thm:main}.
Both proofs are very similar to the proof of Theorem \ref{thm:sub_1} in structure.

\subsection{Almost complete graph} \label{subsec:case_1}

In this subsection, we prove Theorem \ref{thm:almost_clique}.
The first step is to prove the following lemma, which can be
seen as an alternative version of Lemma \ref{lem:rotation_2}.
We intentionally impose a slightly weaker minimum degree condition
of $\frac{n}{6}$ compared to that of Theorem \ref{thm:almost_clique}
with later usage in mind.

\begin{lem} \label{lem:rotation_2_mod}
Suppose that $\beta$ and $\mu$ are reals
satisfying $\beta + \sqrt{\mu} < \frac{1}{1200}$.
Let $G$ be an $n$-vertex graph with minimum degree at least $\frac{n}{6}$
and at least ${n \choose 2} - \beta n^2$ edges, and
$\mathcal{F}$ be a $\mu n$-bounded incompatibility system over $G$.
Then for every edge $e$ of $G$ and smooth path $P$ that contains $e$,
there exists a cycle $C$ compatible with $\mathcal{F}$ with
the following properties:
\begin{itemize}
  \setlength{\itemsep}{0pt} \setlength{\parskip}{0pt}
  \setlength{\parsep}{0pt}
\item $C$ contains $e$,
\item $C$ has length $|C| \ge (\frac{6}{7} - 14\sqrt{\mu})n$,
\item $V(C) \supseteq V(P)$, and
\item $|E(C) \setminus E(P)| \le 2|V(C) \setminus V(P)| + 3$.
\end{itemize}
\end{lem}
\begin{proof}
Let $G$ be a given graph, and $e$ be an edge of $G$.
Let $P$ be a smooth path in $G$ that contains $e$. It suffices to prove
that either there exists a smooth path $P'$ containing $e$ with
$|P'| \ge |P|+1$ and $|E(P') \setminus E(P)| \le 2$, or
a compatible cycle $C$ containing $e$ with $|C| \ge (\frac{6}{7} - 14\sqrt{\mu})n$ 
and $|E(C) \setminus E(P)| \le 3$.
Since then, we can repeatedly find a longer path to eventually
find a cycle with the claimed properties.
Assume that the former event does not occur.

Let $L$ be the
set of vertices that have degree at least $\frac{6}{7}n$ in $G$,
and note that
\[
	\frac{1}{2} \cdot \frac{1}{7}n \cdot |V \setminus L| \le e(G^c) \le \beta n^2.
\]
Hence $|L| \ge (1 - 14\beta)n$. By Lemma \ref{lem:rotation_1}, there exists
a set $X \subset N(v_0) \cap V(P)$ of size
\[
	|X| \ge d(v_0) - 14\sqrt{\mu}n \ge \frac{n}{6} - 14\sqrt{\mu}n > 14\beta n
\]
such that for every vertex $v_i \in X$,
the path $(v_{i-1}, \ldots, v_0, v_i, v_{i+1}, \ldots, v_\ell)$ is smooth.
Thus in particular, we may choose $v_i$ so that $v_{i-1} \in L$.
Let $P' = (w_0, \ldots, w_\ell)$ be the path obtained in this way.
By our assumption on $P$, the path $P'$ cannot be extended by adding one edge.

By Lemma \ref{lem:rotation_1}, there exists a set $Y \subset N(w_0) \cap V(P')$ of size
$|Y| \ge \frac{6}{7}n - 14\sqrt{\mu}n$ such that for every vertex $w_i \in Y$,
the path $(w_{i-1}, \ldots, w_0, w_i, w_{i+1}, \ldots, w_\ell)$ is smooth.
Similarly, there exists a set $Z \subset N(w_\ell) \cap V(P)$ of size
$|Z| \ge \delta(G) - 14\sqrt{\mu}n$ such that for every vertex $w_j \in Z$,
the path $(w_0, w_1, \ldots, w_j, w_\ell, w_{\ell-1}, \ldots, w_{j+1})$ is smooth.
In particular, for each vertex $w_j \in Z$, we see that  $\{w_\ell, w_j\}$
is compatible with both $\{w_j, w_{j-1}\}$ and $\{w_\ell, w_{\ell-1}\}$.
Since 
\[
	|Y| + |Z| 
	\ge \Big(\frac{6}{7}n - 14\sqrt{\mu}n\Big) + (\delta(G) - 14\sqrt{\mu}n)
	\ge \frac{43}{42}n - 28\sqrt{\mu}n
	> n + 2, 
\]
there exists an index $i$
such that $w_{i-1} \in Z$, $w_i \in Y$, and both vertices
$w_{i-1}, w_{i}$ are not incident to $e$. For this index, the cycle
$C = (w_{i-1}, \ldots, w_0, w_i, w_{i+1}, \ldots, w_\ell, w_{i-1})$ is compatible
with $\mathcal{F}$ and contains $e$. Also, $|C| \ge |X| \ge \frac{6}{7}n - 14\sqrt{\mu}n$.
Moreover, $C$ is obtained from $P$ by adding at most three extra edges.
\end{proof}

We now present the proof of Theorem \ref{thm:almost_clique}, which we
restate here for reader's convenience.
\begin{thm*}
Let $\beta$ and $\mu$ be positive reals satisfying
$\beta + 2\sqrt{\mu} \le \frac{1}{1200}$.
Let $G$ be an $n$-vertex graph with minimum degree at least $\frac{n}{5}$ and at least
${n \choose 2} - \beta n^2$ edges, and $\mathcal{F}$ be a $\mu n$-bounded
incompatibility system over $G$.
Then for every edge $e$ of $G$, there exists a Hamilton cycle
containing $e$ that is compatible with $\mathcal{F}$.
\end{thm*}
\begin{proof}
Let $G$ be a given graph, and $e$ be an edge of $G$.
Consider the graph obtained from $G$ by removing all edges whose
two endpoints are $\sqrt{\mu}$-correlated (except $e$).
By Proposition~\ref{prop:bound_bad_correlated},
the resulting graph has minimum degree at least
$\Big(\frac{1}{5} - \sqrt{\mu}\Big)n \ge \frac{n}{6}$, and has at least
${n \choose 2} - (\beta + \sqrt{\mu})n^2$ edges.
By abusing notation, we use $G$ to denote this graph.
Note that Lemma \ref{lem:rotation_2_mod} can be applied to this graph
since $(\beta + \sqrt{\mu}) + \sqrt{\mu} \le \frac{1}{1200}$.

Let $C=(v_0, v_1, \ldots, v_\ell)$ be a  cycle in $G$ compatible
with $\mathcal{F}$, of maximum length. By Lemma \ref{lem:rotation_2_mod},
we have $|C| \ge \Big(\frac{6}{7} - 14\sqrt{\mu}\Big)n $.
Throughout the proof, for a set $X \subset V$,
define $X^+ = \{ v_{i+1} \,:\, v_i \in X \cap V(C)\}$, and
$X^- =  \{ v_{i-1} \,:\, v_i \in X \cap V(C)\}$ (where index
addition and subtraction are modulo $\ell+1$).
If $C$ is a Hamilton cycle, then
we are done. Otherwise, there exists a vertex $z$ not
in the cycle.

Define $B_0$ as the set of vertices that are $2\sqrt{\mu}$-bad
for $C$. By Proposition \ref{prop:bound_bad_correlated},
we know that $|B_0| \le \sqrt{\mu}n$. Hence if $|C| < (1-\sqrt{\mu})n$,
then we may take $z$ to be a vertex not in $B_0$. In this case,
define $B_1$ as the set of bad neighbors of $z$ in $C$. By definition,
we have $|B_1| \le 2\sqrt{\mu}|C|$.
Otherwise if $|C| \ge (1-\sqrt{\mu})n$, then let $z$ be an
arbitrary vertex not in $C$, and define $B_1 = \emptyset$.

Since
\[
	|N(z) \cap V(C)| \ge d(z) + |V(C)| - n
	\ge
	\Big(\frac{n}{5} - \sqrt{\mu}n\Big) + \Big(\frac{6}{7}n - 14\sqrt{\mu}n\Big) - n
	\ge
	7\sqrt{\mu}n
\]
and
\[
	|B_0| + |B_0^+| + |B_0^-| + |B_1| \le 5\sqrt{\mu}n,
\]
the set $T = \{ v_i \in V(C) \,|\, v_i \in N(z), v_{i-1}, v_{i+1} \notin B_0, v_i \notin B_0 \cup B_1 \}$
has cardinality
\[
	|T| \ge |N(z) \cap V(C)| - (|B_0| + |B_0^+| + |B_0^-| + |B_1|) \ge 2\sqrt{\mu}n.
\]
Take a vertex $v_i \in T$ not incident to $e$. Since $\mathcal{F}$ is $\mu n$-bounded,
there are at most $\mu n$ vertices $x \in T$ for which
the pair of edges $\{z, v_i\}$ and $\{z, x\}$ is incompatible.
Also, by Proposition \ref{prop:bound_bad_correlated}, there are
at most $\sqrt{\mu}n$ vertices that are $\sqrt{\mu}$-uncorrelated
with $v_{i+1}$. Therefore, since $|T| \ge 2\sqrt{\mu}n > \mu n + \sqrt{\mu}n + 5$, we
can find a vertex $v_j \in T$ not incident to $e$ and not $v_i, v_{i+1}$, or $v_{i-1}$
for which the pair of edges
$\{z, v_i\}$ and $\{z, v_j\}$ is compatible, and the
pair of vertices $v_{i+1}$ and $v_{j+1}$ is $\sqrt{\mu}$-uncorrelated.
Consider the path
$P = (v_{i+1}, \ldots, v_j, z, v_i, v_{i-1}, \ldots, v_{j+1})$.
First, the two endpoints of $P$ are
$\sqrt{\mu}$-uncorrelated by the choice of $v_i$ and $v_j$.
Second, both $v_{i+1}$ and $v_{j+1}$ have at most
$2\sqrt{\mu}|C| + 3$ bad neighbors in $P$, since
$v_{i+1}, v_{j+1} \notin B_0$, and the set of bad neighbors
in $P$ and in $C$ can differ only in at most three vertices
$v_j, z$, and $v_i$. Third, $P$ contains $e$ since $C$ does,
and $v_i, v_j$ are not incident to $e$.

To check whether $P$ is compatible, it suffices to check the
compatibility of three pairs of edges
$\Big(\{v_{j-1}, v_j\}, \{v_j, z\}\Big)$,
$\Big(\{v_j, z\}, \{z, v_i\}\Big)$, and
$\Big(\{z, v_i\}, \{v_i, v_{i-1}\}\Big)$. The pair of edges
$\{v_j, z\}$ and $\{z, v_i\}$ is compatible by our choice
of $v_i$ and $v_j$.
If $|C| < (1-\sqrt{\mu})n$, then by the choice of $z$ and
the set $B_1$, since $v_i, v_j \notin B_1$ (this follows from 
$v_i, v_j \in T$), we further see
that the other two pairs of edges are both compatible,
thus implying that $P$ is compatible, and therefore smooth.
This by Lemma \ref{lem:rotation_2_mod} gives a compatible cycle
containing $e$ that is longer than $C$, and contradicts
the maximality of $C$.

Therefore, we must have $|C| \ge (1-\sqrt{\mu})n$.
In this case, $P$ is `almost' smoothly compatible with $\mathcal{F}$,
in the sense that it satisfies all the conditions except for possibly the
compatibility of two pairs of edges.
Define $\mathcal{F}_1$ as the incompatibility system obtained from
$\mathcal{F}$ by making the pairs of edges
$\Big(\{v_{j-1}, v_j\}, \{v_j, z\}\Big)$ and
$\Big(\{z, v_i\}, \{v_i, v_{i-1}\}\Big)$ to be compatible.
Note that $P$ is smoothly compatible with $\mathcal{F}_1$.
Hence by Lemma \ref{lem:rotation_2_mod}, we can find a cycle
$C_1$ containing $e$, compatible with $\mathcal{F}_1$, with
$V(C_1) \supseteq V(P) \supseteq V(C)$ and 
\[
	|E(C_1) \setminus E(C)|
	\le
	|E(C_1) \setminus E(P)| + 2
	\le
	2|V(C_1) \setminus V(C)| + 5.
\]

Let $P_1$ be the path obtained from $C_1$ by removing the
edge $\{v_{j-1}, v_j\}$ if it is in $C_1$ (if not, then skip the rest of the paragraph).
We claim that $P_1$ is smoothly compatible with $\mathcal{F}_1$.
First, it is compatible with $\mathcal{F}_1$, since $C_1$ is.
Second, the two endpoints are $\sqrt{\mu}$-uncorrelated,
since we started by removing all edges whose two
endpoints are $\sqrt{\mu}$-correlated.
Third, since $v_j \notin B_0$, we know that
$v_j$ is $2\sqrt{\mu}$-good for $C$. Since $V(C_1) \supseteq V(C)$, it
follows that $v_j$ has at most
\[
	2\sqrt{\mu}|C| + 2|E(C_1) \setminus E(C)| \le 2\sqrt{\mu}|C| + 4(|C_1|-|C|) + 10 \le 6\sqrt{\mu} |C_1| = 6\sqrt{\mu}|P_1|
\]
bad neighbors in $P_1$, where the final inequality follows from
$|C| \ge (1-\sqrt{\mu})n$. A similar estimate holds for the other endpoint $v_{j-1}$.
Let $\mathcal{F}_2$ be the incompatibility system obtained from
$\mathcal{F}$ by making the pair $\Big(\{z, v_i\}, \{v_i, v_{i-1}\}\Big)$ to be compatible.
Note that $P_1$ is smoothly compatible with $\mathcal{F}_2$ as well,
since $P_1$ is smoothly compatible with $\mathcal{F}_1$ and
does not contain the edge $\{v_{j-1}, v_j\}$.
Thus by Lemma \ref{lem:rotation_2_mod}, we can find a cycle $C_2$
containing $e$ that is compatible with $\mathcal{F}_2$,
whose vertex set contains $V(P_1)$. 

Let $P_2$ be the path obtained from $C_2$
by removing the edge $\{v_i, v_{i-1}\}$ if it is in $C_2$
(if not, then it contradicts the maximality of $C$). 
Similarly as before, the path $P_2$
is smoothly compatible with $\mathcal{F}$, and thus by Lemma \ref{lem:rotation_2_mod},
we can find a cycle whose vertex set contains $V(C) \cup \{z\}$,
contradicting the maximality of $C$. Therefore, the cycle
$C$ is a Hamilton cycle.
\end{proof}

\subsection{Almost complete bipartite graph} \label{subsec:case_2}

Let $\mu, \beta, \gamma$ be positive reals satisfying
\[
	\gamma + \beta + \sqrt{\mu} < \frac{1}{2000}.
\]

Let $G$ be a $2m$-vertex bipartite graph with bipartition $A \cup B$
such that $|A| = |B| = m$, with minimum degree at least $\frac{m}{10}$
and at least $(1 - \beta) m^2$ edges. Let $\mathcal{F}$ be a
$\mu m$-bounded incompatibility system defined over $G$.
Further suppose that a perfect matching $e_1 = \{a_1, b_1\}, \ldots, e_m = \{a_m, b_m\}$
satisfying $a_i \in A$ and $b_i \in B$ is given.
Let $f$ be a bijection between $A$ and $B$ defined by the
relation $f(a_i) = b_i$ and $f(b_i) = a_i$ for each $i=1,\ldots,m$.
We will fix these notations throughout the section.
All lemmas and results in this subsection are based
on these notations.

\begin{defn}
A path or a cycle $H$ of $G$ is \emph{proper} if it contains all edges
$e_1, \ldots, e_{\gamma m}$ and
satisfies $f(V(H) \cap A) = V(H) \cap B$.
\end{defn}

We restrict our attention to proper paths and cycles.
The condition $f(V(H) \cap A) = V(H) \cap B$ ensures that the two endpoints
of the path are in $A$ and in $B$, respectively. We consider proper paths because it is a convenient way of forcing such property while using the rotation-extension technique. For example, if for a proper path $P = (v_0, v_1, \ldots, v_\ell)$
the vertex $v_0$ is adjacent to some vertex $x \notin V(P)$, then the path
$(f(x), x, v_0, v_1, \ldots, v_\ell)$ forms a proper path that is longer than $P$. 
To ensure that the new path is compatible with $\mathcal{F}$, we need
the compatibility of the two pairs $(\{f(x), x)\}, \{x, v_0\})$ and
$(\{x, v_0\}, \{v_0, v_1\})$, but it might be the case that there are no 
neighbors of $v_0$ giving the compatibility of these pairs.
Thus we slightly modify the definition of smooth paths.
Define $X_A$ as the set of vertices $x \in A$ for which
there are at least $\sqrt{\mu} m$ indices $i$ such that
the pair of edges $e_i = \{a_i, b_i\}$ and $\{x, b_i\}$ is
incompatible. By counting the number of pairs of edges $\{x, b_i\}$ and
$\{a_i, b_i\}$ that are incompatible in two ways, we obtain
the inequality
\begin{align} \label{eq:bound_x_a}
	|X_A| \cdot \sqrt{\mu}m \le \mu m \cdot m \quad\Longrightarrow\quad |X_A| \le \sqrt{\mu}m. 
\end{align}
Similarly define $X_B \subseteq B$, and we get  $|X_B| \le \sqrt{\mu}m$.
Throughout this section, we will use the following definition of smooth paths.

\begin{defn}
A proper path $P=(v_{0},v_{1},\cdots,v_{\ell})$ with $v_0 \in A$ and $v_\ell \in B$ 
is \emph{smoothly compatible with $\mathcal{F}$} (or \emph{smooth} in short if $\mathcal{F}$ is
clear from the context) if
\vspace{-0.2cm}
\begin{itemize}
  \setlength{\itemsep}{0pt} \setlength{\parskip}{0pt}
  \setlength{\parsep}{0pt}
\item[(i)] $P$ is compatible with $\mathcal{F}$,
\item[(ii)] both endpoints $v_{0}$ and $v_{\ell}$ are $8\sqrt{\mu}$-good
for $P$,
\item[(iii)] $v_0 \notin X_A$ and $v_\ell \notin X_B$.
\end{itemize}
\end{defn}

Note that we no longer impose the two endpoints to be $\sqrt{\mu}$-uncorrelated. 
This is because the pair of vertices $v_0 \in A$ and $v_\ell \in B$ always have no common neighbors 
(recall that the given graph is bipartite). Hence all proper paths automatically satisfy the condition
that the two endpoints are $\sqrt{\mu}$-uncorrelated.

The following modification of Lemma~\ref{lem:rotation_1} will be used.

\begin{lem} \label{lem:rotation_bipartite}
Let $G$ be given as above.
Suppose that $P=(v_0, v_1, \ldots,v_\ell)$ is a proper smooth path in $G$,
where there is no vertex $x \notin V(P)$ for which $(f(x), x, v_0, \ldots, v_\ell)$ is a proper smooth path. Then there exists a set
$Z\subseteq N(v_{0})\cap V(P)$ of size at least
\[ |Z| \ge d(v_0)-(25\sqrt{\mu} + 2\gamma)m \]
such that for every vertex $v_{i}\in Z$, the path
$(v_{i-1},\ldots,v_1,v_{0},v_{i},v_{i+1},\ldots,v_{\ell})$ is a proper smooth path.
\end{lem}
\begin{proof}
Define
\begin{eqnarray*}
B_{0} & = & \Big\{b \in B: \{v_0, b\} \text{ and } \{b, f(b)\} \text{ are incompatible} \Big\} \\
& & \qquad \qquad \qquad \qquad \cup \Big\{b \in B: \{v_\ell, f(b)\} \text{ and } \{b, f(b)\} \text{ are incompatible} \Big\}, \\
B_{1} & = & \Big\{ b \in B \,:\, \{v_0, b\} \mbox{ is incompatible with } \{v_0, v_1\} \mbox{ or } \{v_0, v_{\ell}\} \mbox{ (if exists) } \Big\}, \text{ and}\\
B_{2} & = & \Big\{ a \in A \,:\, a\mbox{ is }2\sqrt{\mu}\mbox{-bad for \ensuremath{P}}\Big\} \cup X_A.
\end{eqnarray*}
We have $|B_0| \le 2\sqrt{\mu}m$ since $v_0 \notin X_A$ and $v_\ell \notin X_B$,
$|B_1| \le 2 \mu m$ since $\mathcal{F}$ is $\mu m$-bounded, and
$|B_2| \le 3\sqrt{\mu} m$ by Proposition \ref{prop:bound_bad_correlated} (with $n = 2m$) and \eqref{eq:bound_x_a}.
Suppose that there exists a neighbor $b$ of $v_0$ such that $b \notin V(P) \cup B_0 \cup B_1 \cup f(B_2)$,
and consider the path $P' = (f(b), b, v_0, v_1, \ldots, v_\ell)$ (note that $f(b) \notin V(P)$ since $P$ is proper). Note that $P'$ is a proper path.
We claim that it in fact is a proper smooth path.
It is compatible since $b \notin B_0 \cup B_1$. The set of bad neighbors of $f(b)$ in
$P'$ is identical to the set of bad neighbors of $f(b)$ in $P$ since $f(b)$ and $v_0$ are
not adjacent.
Similarly, the set of bad neighbors of $v_\ell$ in $P'$ is identical to that in $P$ since
$b \notin B_0 \cup B_1$, and $v_\ell$ is not adjacent to $b$. Therefore, we see that the two endpoints of $P'$ are $8\sqrt{\mu}$-good.
Finally $v_\ell \notin X_B$ since $P$ is a smooth path, and $f(b) \notin X_A$ since $b \notin f(B_2)$.
Therefore $P'$ in fact is a proper smooth path, contradicting our assumption.

Hence all neighbors of $v_0$ are in $V(P) \cup B_0 \cup B_1 \cup f(B_2)$.
Further define
\begin{eqnarray*}
B_{3} & = & \Big\{ w \in V(P) \,:\, w \,\text{ is a bad neighbor of }v_{0}\mbox{ in}\, P, \text{ or intersects some edge } e_1, e_2, \ldots, e_{\gamma m} \Big\}.
\end{eqnarray*}
We have $|B_3| \le 8\sqrt{\mu} |P| + 2\gamma m$, since $v_0$ is $8\sqrt{\mu}$-good for $P$.
Define $B_2^{+} = \{v_{i+1} \,|\, v_i \in B_2 \cap V(P)\}$,
and $Z=\Big(N(v_{0})\cap V(P)\Big)\setminus(B_{1}\cup B_{2}^{+}\cup B_{3})$.
Since all neighbors of $v_0$ are in $V(P) \cup B_0 \cup B_1 \cup f(B_2)$, 
\begin{align*}
	|Z|
	\ge\,&
	|N(v_{0})|- |B_0| - |B_1| - 2|B_2| - |B_3| \\
	\ge\,&
	|N(v_0)| - 2\sqrt{\mu}m - 2\mu m - 6 \sqrt{\mu}m - (8 \sqrt{\mu} \cdot 2m + 2\gamma m)
	\ge
	d(v_0) - 25\sqrt{\mu}m - 2\gamma m\,.
\end{align*}
One can check as in the proof of Lemma~\ref{lem:rotation_1} that
$Z$ satisfies our claim.
\end{proof}

We start by showing that the class of proper smooth paths is
non-empty. This will be achieved in two steps: first proving
the existence of a proper path, and then of a proper smooth path.

\begin{prop} \label{prop:proper_exist}
Let $G$ be a graph given as above.
\begin{itemize}
  \setlength{\itemsep}{0pt} \setlength{\parskip}{0pt}
  \setlength{\parsep}{0pt}
\item[(i)] There exists a proper compatible path $P$ with $|V(P)| \le 10\gamma m - 9$.
\item[(ii)] There exists a proper smooth path $P'$ with $|V(P')| \le 10\gamma m$.
\end{itemize}
\end{prop}
\begin{proof}
Let $Y_A \subseteq A$ be the set of vertices in $A$ of degree less than
$\frac{3}{4}m$. Since
\[
	\frac{m}{4} \cdot |Y_A| \le e_{G^c}(A,B) \le \beta m^2,
\]
we see that $|Y_A| \le 4\beta m$.
Similarly define $Y_B \subseteq B$, and we get $|Y_B| \le 4\beta m$.

\medskip

\noindent (i) We prove the following statement
for $t=1,\ldots, \gamma m$ using induction on $t$: there exists a compatible path $P_t$ 
of length at most $10t - 9$ containing the edges $e_1, \ldots, e_t$
and satisfying $f(V(P_t) \cap A) = V(P_t) \cap B$.
The statement is trivially true for $t=1$. Suppose that we are
given a compatible path $P_t$ as above.
We may assume that $P_t$ does not contain $e_{t+1} = \{a_{t+1}, b_{t+1}\}$
as otherwise the induction step trivially holds.

Let $b \in B$ be an endpoint of $P_t$, and let $\{a,b\}$ be an edge of $P_t$
incident to $b$.
Let $A_1 \subseteq A$ be the neighbors $a_i$ of $b$
with the following properties:
(i) $\{b, a_i\}$ is compatible with $\{a,b\}$,
(ii) $a_i \notin X_A \cup Y_A$ and $b_i = f(a_i) \notin X_B \cup Y_B$, and
(iii) $a_i, b_i \notin V(P_t)\cup \{a_{t+1}, b_{t+1}\}$.
Note that
\begin{align*}
	|A_1|
	&\ge
	|N(b)| - \mu m - (|X_A| + |X_B|) - (|Y_A| + |Y_B|) - (|V(P_t)| + 2)  \\
	&\ge
	\frac{1}{10} m - \mu m- 2\sqrt{\mu}m - 8\beta m  - 10\gamma m
	> \frac{1}{11}m.
\end{align*}
Let $B_1 \subseteq B$ be the neighbors $b_j$ of $a_{t+1}$
with the following properties:
(i) $\{a_{t+1}, b_j\}$ is compatible with $e_{t+1}$,
(ii) $a_j \notin X_A \cup Y_A$ and $b_j \notin X_B \cup Y_B$, an
(iii) $a_j, b_j \notin V(P_t) \cup \{a_{t+1}, b_{t+1}\}$. 
A computation similar to above shows that $|B_1| > \frac{1}{11}m$.

Since $e(A,B) \ge m^2 - \beta m^2 > m^2 - (|A_1||B_1| - m)$,
there exists an edge
$\{b_i, a_j\}$ such that $b_i \in f(A_1), a_j \in f(B_1)$ and $i \neq j$.
Our goal is to find two indices $k$ and $\ell$ for which
the path 
$$P_{t+1} = (P, b, a_i, b_k, a_k, b_i, a_j, b_\ell, a_\ell, b_j, a_{t+1}, b_{t+1})$$
is compatible.
By the definitions of $b_i$ and $a_j$,
it suffices to show the existence of distinct indices $k$ and $\ell$ for which
the path $(b, a_i, b_k, a_k, b_i, a_j)$ and the path
$(b_i, a_j, b_\ell, a_\ell, b_j, a_{t+1})$ are both compatible.

The compatibility of the path $(b, a_i, b_k, a_k, b_i, a_j)$ depends
only on the index $k$. We must first have $a_i$ adjacent to $b_k$
and $b_i$ adjacent to $a_k$, and avoid having $a_k,b_k$ in the
set $V(P_t)$ or in $\{a_i, b_i, a_j, b_j, a_{t+1}, b_{t+1}\}$.
Since $a_i \notin Y_A$ and $b_i \notin Y_B$, the number of possible indices satisfying
the restriction is at least
\[
	\frac{3}{4}m + \frac{3}{4}m - m - (|V(P_t)| + 6)
	\ge
	\frac{1}{2}m - 10\gamma m - 6.
\]
Moreover, since $a_i \notin X_A$ and $b_i \notin X_B$, 
the compatibility of the pairs of edges further forbid
$2\sqrt{\mu}m + 2\mu m$ indices $k$.
Thus we can find an index $k$ for which the path $(b, a_i, b_k, a_k, b_i, a_j)$
is compatible. Similarly, we can find an index $\ell \neq k$ for which
the path $(b_i, a_j, b_\ell, a_\ell, b_j, a_{t+1})$ is
compatible. Note that for this choice of $k$ and $\ell$,
the path $P_{t+1}$ satisfies $|P_{t+1}| = |P_t| + 10$ and
$f(V(P_{t+1}) \cap A) = V(P_{t+1}) \cap B$. This completes the proof
of the inductive step.

\medskip

\noindent (ii) By part (i), there exists a
proper compatible path $P$ with $|V(P)| \le 10\gamma m - 9$.
Let $a$ and $b$ be the two endpoints of $P$,
where $a \in A$ and $b \in B$. Let $B'$ be
the set of vertices $x \in N(a) \setminus (V(P) \cup X_B \cup Y_B)$ that are
(i) connected to $a$ by an edge compatible
with the edge incident to $a$ in $P$, for which
(ii) $f(x) \notin X_A \cup Y_A$, and (iii) $f(x)$ is $2\sqrt{\mu}$-good for $P$.
By Proposition~\ref{prop:bound_bad_correlated} (with $n = 2m$),
\begin{align*}
	|B'|
	&\ge
	|N(a)| - (|V(P)| + |X_B| + |Y_B|) - \mu m - (|X_A| + |Y_A|) - 2\sqrt{\mu} m \\	
	&\ge
	\frac{m}{10} - (10\gamma m + \sqrt{\mu} m + 4\beta m) - \mu m - (4\beta m + \sqrt{\mu}m) - 2\sqrt{\mu}m
	> 0.
\end{align*}
Let $b'$ be a vertex in $B'$, and let $a' = f(b')$.
Let $I$ be the set of indices $i$ such that
$(b, P, a, b', a_i, b_i, a')$ is
a compatible path. Since $b' \notin X_B \cup Y_B$ and $a' \notin X_A \cup Y_A$, we have
\[
	|I| \ge \frac{3}{4}m + \frac{3}{4}m - m - 2\sqrt{\mu}m - \mu m > 0.
\]
Fix an arbitrary index $i \in I$. 
Similarly as above,
let $A'$ be the set of vertices $y \in N(b) \setminus (V(P) \cup X_A \cup Y_A \cup \{a,a'\})$ that are (i) connected
to $b$ by an edge compatible with the edge incident to $b$ in $P$ and
for which (ii) $f(y) \notin X_B \cup Y_B$, and
(iii) $f(y)$ is $2\sqrt{\mu}$-good for $P$.
By Proposition \ref{prop:bound_bad_correlated} (with $n = 2m$), we have
\[
	|A'| \ge |N(b)| - (|V(P)| + |X_A| + |Y_A| + 2) - \mu m - (|X_B| + |Y_B|) - 2\sqrt{\mu} m > 0.
\]
Let $a''$ be a vertex in $A'$, and let $b'' = f(a'')$.
Similarly as above, we can find an index $j \neq i$ such that
the path $P' = (b'', a_j, b_j, a'', b, P, a, b', a_i, b_i, a')$
is compatible and proper. 
To show that $P'$ is smooth, it suffices to show that the two
endpoints are $8\sqrt{\mu}$-good for $P'$. This easily follows
from the fact that $a'$ and $b''$ are $2\sqrt{\mu}$-good for $P$
and that $|V(P') \setminus V(P)| = 8$ (note that $|V(P)| \ge \gamma m$).
\end{proof}

The following lemma is a variant of Lemmas \ref{lem:rotation_2} and
\ref{lem:rotation_2_mod}.

\begin{lem} \label{lem:rotation_almost_bipartite}
Let $G$ be given as above.
For every proper smooth path $P$, there exists a proper compatible cycle $C$
satisfying the following properties:
\begin{itemize}
  \setlength{\itemsep}{0pt} \setlength{\parskip}{0pt}
  \setlength{\parsep}{0pt}
\item $C$ has length $|C| \ge (\frac{19}{10} - 50\sqrt{\mu} - 4\gamma)m$,
\item $V(C) \supseteq V(P)$, and
\item $|E(C) \setminus E(P)| \le \frac{3}{2}|V(C) \setminus V(P)| + 3$.
\end{itemize}
\end{lem}
\begin{proof}
Let $G$ be a given graph.
Let $P = (v_0, \ldots, v_\ell)$ be a proper smooth path in $G$, where $v_0 \in A$ and $v_\ell \in B$. It suffices to prove
that either there exists a proper smooth path $P'$ with
$|P'| \ge |P|+2$ and $|E(P') \setminus E(P)| \le 3$, or
a proper compatible cycle $C$ with $|C|\ge (\frac{19}{10} - 50\sqrt{\mu} - 4\gamma)m$ 
and $|E(C) \setminus E(P)| \le 3$.
Since then, we can repeatedly find a longer path to eventually
find a cycle with the claimed properties.
Assume that the former event does not occur.

Let $L_A \subseteq A$ be the
set of vertices in $A$ that have degree at least $\frac{19}{20}m$ in $G$,
and note that
\[
	\frac{1}{20}m \cdot |V \setminus L_A| \le e_{G^c}(A, B) \le \beta m^2.
\]
Hence $|L_A| \ge m - 20\beta m$. By Lemma \ref{lem:rotation_bipartite}, there exists
a set $X \subset N(v_0) \cap V(P)$ of size
\[
	|X| \ge d(v_0) - 25\sqrt{\mu}m  - 2\gamma m > 20 \beta m
\]
such that for all vertices $v_i \in X$,
the path $(v_{i-1}, \ldots, v_0, v_i, v_{i+1}, \ldots, v_\ell)$ is proper and smooth.
Since $|X| > m - |L_A|$, there exists a vertex $v_i \in X$ for which $v_{i-1} \in L_A$.
Let $P' = (w_0, w_1, \ldots, w_\ell)$ be the proper smooth path obtained by taking
$v_i$ as a pivot point (where $w_0 \in L_A$ and $w_\ell \in B$).

By our assumption on $P$, we know that $P'$ cannot be extended into
a longer proper smooth path by adding at most two edges.
Hence by Lemma \ref{lem:rotation_bipartite}, since $w_0 \in L_A$,
there exists a set $Y \subset N(w_0) \cap V(P')$ of size
$|Y| \ge \frac{19}{20}m - (25\sqrt{\mu} + 2\gamma)m$ such that for all vertices $w_i \in Y$,
the path $(w_{i-1}, \ldots, w_0, w_i, w_{i+1}, \ldots, w_\ell)$ is proper and smooth.
If $w_\ell \in Y$, then we immediately find a cycle with the claimed properties, 
and hence we may assume that $w_\ell \notin Y$. 
Then for each vertex $w_i \in Y$, we see that the
edge $\{w_0, w_i\}$ is compatible with both
$\{w_{0}, w_{1}\}$ and $\{w_{i}, w_{i+1}\}$, and that $\{w_{i-1}, w_{i}\}$ is
not one of the edges $e_1, \ldots, e_{\gamma m}$.
Similarly, there exists a set $Z \subset N(w_\ell) \cap V(P)$ of size
$|Z| \ge \delta(G) - (25\sqrt{\mu}+ 2\gamma)m$ such that for all vertices $w_j \in Z$,
the path $(w_0, w_1, \ldots, w_j, w_\ell, w_{\ell-1}, \ldots, w_{j+1})$ is proper and smooth.
In particular, for each vertex $w_j \in Z$, we see that the
edge $\{w_j, w_\ell\}$ is compatible with both
$\{w_\ell, w_{\ell-1}\}$ and $\{w_j, w_{j-1}\}$.

Since $G$ is a bipartite graph, $w_0$ is adjacent only to
vertices $w_i$ with odd index $i$, and $w_\ell$ is adjacent only to vertices
$w_i$ with even index $i$. Therefore, since
\[
	|Y| + |Z| \ge \Big(\frac{19}{20}m - (25\sqrt{\mu}+2\gamma)m\Big) + \Big(\frac{m}{10} - (25\sqrt{\mu}+2\gamma)m\Big) > m,
\]
there exists an index $i$
such that $w_{i-1} \in Z$ and $w_i \in Y$. For this index $i$, the cycle
$C = (w_{i-1}, \ldots, w_0, w_i, w_{i+1}, \ldots, w_\ell, w_{i-1})$
is a proper cycle compatible with $\mathcal{F}$. Also,
\[
	|C| \ge 2|Y|
	\ge
	2 \cdot \Big( \frac{19}{20}m - (25\sqrt{\mu}+2\gamma)m\Big).
\]
Moreover, $C$ is obtained from $P$ by adding at most three extra edges.
\end{proof}

We now present the proof of Theorem \ref{thm:almost_bipartite},
which asserts the existence of
a proper Hamilton cycle compatible with $\mathcal{F}$.

\begin{proof}[Proof of Theorem \ref{thm:almost_bipartite}]
Let $G$ be a graph satisfying the conditions given
in Theorem \ref{thm:almost_bipartite}.
Let $C=(v_0, v_1, \ldots, v_\ell)$ be a proper cycle
in $G$ compatible with $\mathcal{F}$, of maximum length.
The existence of such a cycle follows from Proposition \ref{prop:proper_exist}
and Lemma \ref{lem:rotation_almost_bipartite}. Moreover, Lemma \ref{lem:rotation_almost_bipartite} 
shows that $|C| \ge (\frac{19}{10} - (50\sqrt{\mu}+4\gamma))m$.
Throughout the proof, for a set $X \subset V$,
define $X^+ = \{ v_{i+1} \,:\, v_i \in X \cap V(C)\}$, and
$X^- =  \{ v_{i-1} \,:\, v_i \in X \cap V(C)\}$ (where index
addition and subtraction are modulo $\ell+1$).
If $C$ is a Hamilton cycle, then
we are done. Otherwise, there exists an edge $e$ in the matching,
not intersecting the cycle.

Define $B_0$ as the set of vertices that are $2\sqrt{\mu}$-bad
for $C$. By Proposition \ref{prop:bound_bad_correlated} (with $n = 2m$),
we know that $|B_0| \le 2\sqrt{\mu}m$. Hence if $|C| < (2-4\sqrt{\mu})m$,
then we may take an edge $e = \{a,b\}$ in the matching
so that $a,b \notin B_0$. In this case,
define $B_a$ as the set of bad neighbors of $a$, and $B_b$ as the
set of bad neighbors of $b$. By definition,
we have $|B_a|, |B_b| \le 2\sqrt{\mu}|C|$.
Otherwise if $|C| \ge (2-4\sqrt{\mu})m$, then let $e = \{a,b\}$ be an
arbitrary edge of the matching not intersecting $C$,
and define $B_a = B_b = \emptyset$ (in both cases,
we assume that $a \in A$ and $b \in B$).

Since $G$ is a balanced bipartite graph, we have
\[
	|N(a) \cap V(C)| \ge d(a) + \frac{1}{2}|V(C)| - m
	\ge
	\frac{1}{10}m + \frac{1}{2}\Big(\frac{19}{10} - (50\sqrt{\mu}+4\gamma)\Big)m - m
	\ge
	(3\gamma + 14\sqrt{\mu})m.
\]
Define $T_a = \{ v_i \in N(a) \cap V(C) \,:\, v_{i-1}, v_{i+1} \notin X_A \cup B_0, v_i \notin X_B \cup B_0 \cup B_a \}$, and note that since
\[
	2(|X_A| + |B_0|) + |X_B| + |B_0| + |B_a| \le 13\sqrt{\mu}m,
\]
we have $|T_a| \ge (3\gamma + \sqrt{\mu})m$.
We can similarly define a set $T_b \subseteq N(b) \cap V(C)$ of size $|T_b| \ge (3\gamma + \sqrt{\mu})m$.

Take a vertex $v_i \in T_a$ not incident to $e_1, \ldots, e_{\gamma m}$,
for which the pair of edges
$\{a,v_i\}$ and $\{a,b\}$ is compatible (such vertex exists since
$|T_a| \ge (3\gamma + \sqrt{\mu})m$ and $\mathcal{F}$ is $\mu m$-bounded).
Then similarly take a vertex $v_j \in T_b$ not incident to $e_1, \ldots, e_{\gamma m}$,
for which the pair of edges
$\{a, b\}$ and $\{b, v_j\}$ is compatible and $v_j \neq v_{i-1}, v_{i+1}$.
Consider the path
$P = (v_{i+1}, v_{i+2}, \ldots, v_j, b, a, v_i, v_{i-1}, \ldots, v_{j+1})$.
First, $P$ is a proper path since $v_i$ and $v_j$ are not incident to
$e_1, \ldots, e_{\gamma m}$.
Second, $v_{i+1} \notin X_A$ and $v_{j+1} \notin X_B$ by the definitions
of $T_a$ and $T_b$,
and third, both $v_{i+1}$ and $v_{j+1}$ have at most
$2\sqrt{\mu}|C| + 4$ bad neighbors in $P$, since
$v_{i+1}, v_{j+1} \notin B_0$, and the set of bad neighbors
in $P$ and in $C$ can differ only in at most four vertices
$v_j, a,b$, and $v_i$.

To check if $P$ is compatible, it suffices to check the
compatibility of four pairs
$\Big(\{v_{j-1}, v_j\}, \{v_j, b\}\Big)$,
$\Big(\{v_j, b\}, \{b, a\}\Big)$,
$\Big(\{b, a\}, \{a, v_i\}\Big)$, and
$\Big(\{a, v_i\}, \{v_i, v_{i-1}\}\Big)$.
The two pairs of edges
$\Big(\{v_j, b\}, \{b, a\}\Big)$ and
$\Big(\{b, a\}, \{a, v_i\}\Big)$ are both compatible by our choice
of $v_i$ and $v_j$.
If $|C| < (2-4\sqrt{\mu})m$, then by the choice of $e$ and
of the sets $B_a, B_b$, since $v_i \notin B_a$ and $v_j \notin B_b$,
we further see that the other pairs of edges are both compatible,
thus implying that $P$ is compatible, and is therefore smooth.
This by Lemma \ref{lem:rotation_almost_bipartite}
gives a proper compatible cycle longer than $C$ and contradicts
the maximality of $C$.
Therefore, we must have $|C| \ge (2-4\sqrt{\mu})m$.

In this case, $P$ is `almost' smoothly compatible with $\mathcal{F}$,
in the sense that it satisfies all the conditions except for possibly the
compatibility of two pairs of edges.
Define $\mathcal{F}_1$ as the incompatibility system obtained from
$\mathcal{F}$ by making the pairs of edges
$\Big(\{v_{j-1}, v_j\}, \{v_j, b\}\Big)$ and
$\Big(\{a, v_i\}, \{v_i, v_{i-1}\}\Big)$ to be compatible.
Note that $P$ is smoothly compatible with $\mathcal{F}_1$.
Hence by Lemma \ref{lem:rotation_almost_bipartite}, we can find a proper cycle
$C_1$ compatible with $\mathcal{F}_1$, with
$V(C_1) \supseteq V(P) \supseteq V(C)$ and 
\[
	|E(C_1) \setminus E(C)|
	\le
	|E(C_1) \setminus E(P)| + 3
	\le
	\frac{3}{2}|V(C_1) \setminus V(C)| + 6.
\]

Let $P_1$ be the path obtained from $C_1$ by removing the
edge $\{v_{i-1}, v_i\}$ if it is in $C_1$ (if not, then skip this
paragraph). Note that $P_1$ is a proper
path. We claim that $P_1$ is smoothly compatible with $\mathcal{F}_1$.
First, it is compatible with $\mathcal{F}_1$, since $C_1$ is.
Second, $v_{i-1} \notin X_A$ and $v_i \notin X_B$ since
$v_i \in T_a$.
Third, since $v_i \notin B_0$, we know that
$v_i$ is $2\sqrt{\mu}$-good for $C$. Since $V(C_1) \supseteq V(C)$
and $G$ is bipartite, it follows that $v_j$ has at most
\[
	2\sqrt{\mu}|C| + |E(C_1) \setminus E(C)| \le 2\sqrt{\mu}|C| + \frac{3}{2}(|C_1|-|C|) + 6 \le 8\sqrt{\mu} |C_1| = 8\sqrt{\mu}|P_1|
\]
bad neighbors in $P_1$, where the second inequality follows from
$|C| \ge (2-4\sqrt{\mu})m$. A similar estimate holds for the other endpoint $v_{i-1}$.
Let $\mathcal{F}_2$ be the incompatibility system obtained from
$\mathcal{F}$ by making the pair $\Big(\{a, v_j\}, \{v_j, v_{j-1}\}\Big)$ to be compatible.
Note that $P_1$ is compatible with $\mathcal{F}_2$,
since $P_1$ is compatible with $\mathcal{F}_1$ and does not
contain the edge $\{v_{j-1}, v_j\}$.
By Lemma \ref{lem:rotation_almost_bipartite}, there exists a proper cycle $C_2$
compatible with $\mathcal{F}_2$ whose vertex set contains $V(C) \cup \{a,b\}$. 

Let $P_2$ be the path obtained from $C_2$
by removing the edge $\{v_i, v_{i-1}\}$ if it is in $C_2$. 
An argument similar to above shows that $P_2$ is a proper path
smoothly compatible with $\mathcal{F}$. By Lemma \ref{lem:rotation_2},
we can find a proper cycle compatible with $\mathcal{F}$ whose vertex set
contains $V(C) \cup \{a, b\}$,
contradicting the maximality of $C$. Therefore, the cycle
$C$ is a Hamilton cycle.
\end{proof}

\section{Concluding remarks} \label{sec:conclusion}

\noindent $\bullet$ We have proven the existence of a constant $\mu>0$ such that the following holds
for large enough $n$: for every $n$-vertex Dirac graph $G$
with a given $\mu n$-bounded incompatibility system $\mathcal{F}$,
there exists a Hamilton cycle in $G$ compatible with $\mathcal{F}$.
The value of $\mu$ that we obtain is quite small ($\mu = 10^{-16}$),
and determining the best possible value of $\mu$ is an interesting open problem
remaining to be solved.
It is not clear what this value should be. The following variant of a
construction of Bollob\'as and Erd\H{o}s \cite{BoEr76}
shows that $\mu$ is at most $\frac{1}{4}$. Let $n$ be an integer of the form
$4k-1$, and let $G$ be an edge-disjoint union of
two $\frac{n+1}{4}$-regular graphs $G_1$ and $G_2$ on the
same $n$-vertex set. Color the edges
of $G_1$ in red, and of $G_2$ in blue. Note that $G$ does not contain
a properly colored Hamilton cycle since a Hamilton cycle of $G$
is of odd length. Let $\mathcal{F}$
be an incompatibility system defined over $G$, where incident edges
of the same color are incompatible. Then there exists a
Hamilton cycle compatible with $\mathcal{F}$ if and only if
there exists a properly colored Hamilton cycle. Since there is no
properly colored Hamilton cycle,
we see that there is no Hamilton cycle compatible with $\mathcal{F}$.

\medskip

\noindent $\bullet$ As mentioned in the introduction, the motivation for our work came
from a conjecture of H\"aggkvist (Conjecture \ref{conj:haggkvist}).
We note that the conjecture can be answered using a result
in \cite{KrLeSu14} that studied
Hamiltonicity Maker-Breaker game played on Dirac graphs. The theorem
proven there asserts the existence of a positive constant $\beta$ such that
Maker has a winning strategy in a $(1:\beta n/\log n)$
Hamiltonicity Maker-Breaker game played on Dirac graphs.
To see how this implies the conjecture, given a graph $G$ and
a $1$-bounded incompatibility system
$\mathcal{F}$, consider a Breaker's strategy
claiming at each turn the edges that are incompatible with
the edge that the Maker claimed in the previous turn; this strategy
forces Maker's graph to be compatible
with $\mathcal{F}$ at all stages. Since Maker has
a winning strategy for a $(1:2)$ game, we see that there exists a
Hamilton cycle compatible with $\mathcal{F}$. This analysis
gives a weaker version of our main theorem asserting the existence
of a compatible Hamilton cycle for every $\frac{1}{2}\beta n/\log n$-bounded
incompatibility system.

\medskip

\noindent $\bullet$
The concept of incompatibility systems appears to provide a new and interesting
take on robustness of graph properties. Further study of how various extremal results
can be strengthened using this notion appears to be a promising direction of research.
For example in the forthcoming paper \cite{KrLeSu_other}, we  show that there exists a constant $\mu>0$
such that with high probability over the choice of a random graph $G=G(n,p)$ with $p \gg \frac{\log n}{n}$,  for
any $\mu np$-bounded system $\mathcal{F}$ over $G$, there is a compatible
Hamilton cycle. This extends classical Hamiltonicity results of random graphs.

\vspace{0.4cm}
\noindent 
{\bf Acknowledgement.} A major part of this work was carried out when Benny Sudakov was visiting Tel Aviv University, Israel. 
He would like to thank the School of Mathematics in Tel Aviv University for hospitality and for creating a stimulating research
environment.

\end{document}

%% file: fig-rotatepath.tex
\begin{tikzpicture}

\def\xendtwo{8 cm}
\def\yshifttwo{2.2 cm}

   \foreach \x in {0.0,0.8,...,8.8} {
            \draw [fill=black] (\x cm, \yshifttwo) circle (0.5mm);
        }

   \foreach \x in {0.0,0.8,1.6} {
            \draw (\x cm,\yshifttwo) -- (\x cm + 0.8cm,\yshifttwo);
        }
   \draw [dotted] (2.4 cm,\yshifttwo) -- (3.2cm,\yshifttwo);

   \foreach \x in {3.2,4.0,...,8.0} {
            \draw (\x cm,\yshifttwo) -- (\x cm + 0.8cm,\yshifttwo);
        }
\

    \draw (0, \yshifttwo) .. controls (1.2cm,\yshifttwo + 0.7cm)
            and (2.0cm,\yshifttwo + 0.7cm) .. (3.2cm, \yshifttwo);
    \draw (3.3cm, \yshifttwo - 0.4cm) node {$v_{i}$};
    \draw (2.46cm, \yshifttwo - 0.4cm) node {$v_{i-1}$};

    \draw (0.0cm, \yshifttwo - 0.4cm) node {$v_{0}$};
    \draw (8.0cm, \yshifttwo - 0.4cm) node {$v_{\ell}$};

\end{tikzpicture}

%% file: fig-twicerotate.tex
\begin{tikzpicture}

\def\xend{11}

\def\yshiftone{2.2 cm}

   \draw (0 cm, \yshiftone) -- (2.4 cm, \yshiftone);   
   \draw [-triangle 45] (2.4 cm, \yshiftone) -- (3.8 cm, \yshiftone);
   \draw (3.8 cm, \yshiftone) -- (7.6 cm, \yshiftone);
   
   \draw [dotted] (7.6 cm,\yshiftone) -- (8.1 cm,\yshiftone);

   \draw (8.1cm,\yshiftone) -- (9.0cm, \yshiftone);
   \draw [dotted] (9.0 cm,\yshiftone) -- (9.6 cm,\yshiftone);

   \draw (9.6cm,\yshiftone) -- (\xend, \yshiftone);

	\draw (0 cm,\yshiftone + 0.1cm) -- (0 cm,\yshiftone - 0.12 cm) 
							-- (7.7 cm, \yshiftone - 0.12cm) -- (7.7cm,\yshiftone + 0.1cm);
	\draw (3.8 cm, \yshiftone - 0.4cm) node {$I_1$};

	\draw (8.0 cm,\yshiftone + 0.1cm) -- (8.0 cm,\yshiftone - 0.12 cm) 
							-- (\xend, \yshiftone - 0.12cm) -- (\xend,\yshiftone + 0.1cm);
	\draw (9.8 cm, \yshiftone - 0.4cm) node {$I_2$};	

    \draw (0, \yshiftone) .. controls (2.7cm,\yshiftone + 1.1cm)
            and (5.4cm,\yshiftone + 1.1cm) .. (8.1cm, \yshiftone);
    \draw (7.6cm, \yshiftone) .. controls (8.3cm,\yshiftone + 0.5cm)
            and (8.9cm,\yshiftone + 0.5cm) .. (9.6cm, \yshiftone);

   \draw [fill=black] (0 cm,\yshiftone) circle (0.5mm);
   \draw [fill=black] (7.6 cm,\yshiftone) circle (0.5mm);
   \draw [fill=black] (9.0 cm, \yshiftone) circle (0.5mm);
   \draw [fill=black] (\xend cm , \yshiftone) circle (0.5mm);

    \draw (0 cm, \yshiftone - 0.4cm) node {$v_{0}$};
    \draw (9.0cm, \yshiftone - 0.4cm) node {$w$};
    \draw (\xend cm, \yshiftone - 0.4cm) node {$v_{\ell}$};

\def\yshifttwo{0 cm}

   \draw (0 cm, \yshifttwo) -- (3.8 cm, \yshifttwo);   
   \draw [-triangle 45] (4.5 cm , \yshifttwo) -- (3.8 cm, \yshifttwo);
   \draw (4.5 cm, \yshifttwo) -- (9.0 cm, \yshifttwo);
   \draw [dotted] (9.0 cm,\yshifttwo) -- (9.6 cm,\yshifttwo);   
   \draw (9.6cm,\yshifttwo) -- (\xend, \yshifttwo);

	\draw (0 cm,\yshifttwo + 0.1cm) -- (0 cm,\yshifttwo - 0.12 cm) 
							-- (7.7 cm, \yshifttwo - 0.12cm) -- (7.7cm,\yshifttwo + 0.1cm);
	\draw (3.8 cm, \yshifttwo - 0.4cm) node {$I_1$};

	\draw (8 cm,\yshifttwo + 0.1cm) -- (8 cm,\yshifttwo - 0.12 cm) 
							-- (\xend, \yshifttwo - 0.12cm) -- (\xend,\yshifttwo + 0.1cm);
	\draw (9.8 cm, \yshifttwo - 0.4cm) node {$I_2$};

    \draw (0, \yshifttwo) .. controls (3.5cm,\yshifttwo + 1.1cm)
            and (6.1cm,\yshifttwo + 1.1cm) .. (9.6cm, \yshifttwo);

   \draw [fill=black] (0 cm,\yshifttwo) circle (0.5mm);
   \draw [fill=black] (9.0 cm, \yshifttwo) circle (0.5mm);
   \draw [fill=black] (\xend cm , \yshifttwo) circle (0.5mm);

    \draw (0 cm, \yshifttwo - 0.4cm) node {$v_{0}$};            
    \draw (9.0cm, \yshifttwo - 0.4cm) node {$w$};
    \draw (\xend cm, \yshifttwo - 0.4cm) node {$v_{\ell}$};

\end{tikzpicture}